\documentclass[12pt]{amsart}

 \usepackage{fullpage}
\usepackage{amsmath,amssymb}
\usepackage{datetime}
\usepackage{graphicx}
\usepackage{appendix}
\usepackage[utf8]{inputenc}
\usepackage{bbm}

\newtheorem{thm}{Theorem}[section]
\newtheorem{lem}[thm]{Lemma}
\newtheorem{prop}[thm]{Proposition}
\newtheorem{rem}{Remark}
\newtheorem*{ex}{Example}

\def\sumstar{\sideset{}{^*}\sum}
\newcommand{\FN}{F_N} 
\newcommand{\Z}{{\mathbb Z}}

\newcommand{\Q}{{\mathbb Q}}
\newcommand{\R}{{\mathbb R}}
\newcommand{\OO}{{\mathcal O}}
\newcommand{\E}{{\mathbb E}}
\newcommand{\T}{{\mathbb T}} 
\newcommand{\one}{\mathbbm{1}}

\newcommand{\tmop}[1]{\ensuremath{\operatorname{#1}}}
\newcommand{\tmod}[1]{\ (\tmop{mod}\ #1)}

\def\pamod{\! \! \! \! \pmod}
\newcommand{\Mod}[1]{\ (\mathrm{mod}\ #1)}

\numberwithin{equation}{section}

  \newcommand{\fixmepk}[1]{}

\title{Poisson spacing statistics for lattice points on circles}

\author{P\"ar Kurlberg and Stephen Lester}
 \address{Department of Mathematics, KTH, 
 SE-100 44 Stockholm, Sweden}
 \email{kurlberg@math.kth.se}
\address{Department of Mathematics, King's College London, London WC2R 2LS, UK}
 \email{steve.lester@kcl.ac.uk}
\date{\today}

\begin{document}
\begin{abstract}
  We show that along a density one subsequence of admissible radii,
  the nearest neighbor spacing between lattice points on circles is
  Poissonian.
\end{abstract}
\maketitle

\section{Introduction}


For a sequence of real numbers $(a_n)$  and integer $N \ge 1$ the probability measure given by
\[
\nu_N= \frac{1}{N} \sum_{1 \le n \le N} \delta( \{ a_n\}),
\]
where $\delta(\cdot)$ denotes the Dirac delta function and $\{a_n\}$
is the fractional part of $a_n$, describes the distribution of the
points $( a_n)_{n=1}^N$ within $\mathbb T=\mathbb R/\mathbb Z$ and we
say that $( \{a_n\})$ is uniformly distributed modulo $1$ if $\nu_N$
weakly converges to the Lebesgue measure $\lambda$ on $\mathbb T$ as
$N \rightarrow \infty$.  A classical result of Weyl provides the
following criterion: $( \{a_n\})$ is uniformly distributed if and only
if the Fourier coefficients of $\nu_N$ converge pointwise to those of
$\lambda$ as $N \rightarrow \infty$. 

A more refined study of the behavior of a sequence modulo $1$ examines the local spacing statistics of the sequence, that is how the sequence is distributed at the scale $1/N$, which
provides insight into how the elements of the sequence
are spaced together. 
Computing the local spacing statistics of a given sequence is in general a challenging problem. Two important examples include the nearest neighbor spacing statistics of the sequence
consisting of the ordinates of zeros of the Riemann zeta-function as well as the sequence of
energy levels of quantized Hamiltonians; both of these are subjects of well-known open conjectures. On the other hand
there are some notable instances where the local spacing statistics
have, partially or fully, been computed, e.g. see
\cite{Hooley1,Hooley2, Hooley3, 
  rudnick-zaharescu, chaubey-yesha, kurlberg-rudnick, boca-zaharescu,
  BCZ01, ABCZ}, including works on angles of Euclidean lattice points
visible from the origin \cite{BCZ00,marklof-strombergsson}.

In this article we study the spacing of angles arising from
$\Z^2$-lattice points lying on circles.  Due to the symmetry of lattice
points under rotation by $\pi/2$ it suffices to consider their angles
modulo $\pi/2$ and given a circle of radius $\sqrt{n}$ we write
$r(n)=\tfrac14\#\{ \vec x \in \mathbb Z^2: |\vec x|^2=n\}$.  For
$\vec x=(x,y) \in \mathbb Z^2$ we write
$\theta_{\vec x}:= \tmop{arg}(x+iy) \Mod{\pi/2}$
and $\mathcal S$ for the set of natural numbers, which can be written
as a sum of two integer  squares. Given $n \in \mathcal S$ with $r(n)=N$, we
define the probability measure $\varrho_N$ on $[0,\pi/2)$ by 
\[
\varrho_N=\frac{1}{4N} \sum_{\substack{\vec x \in \mathbb Z^2 \\ |\vec x|^2=n}} \delta(\theta_{\vec x}).
\]
A striking result independently due to K\'{a}tai and K\"{o}rnyei
\cite{katai-kornyei} and Erd\H{o}s and Hall \cite{erdos-hall} shows
that there exists a density one subsequence 
$\mathcal S_0 \subset \mathcal S$ such that the angles
$\theta_{\vec x}$ with $|\vec x|^2 \in \mathcal S_0$ very strongly
equidistribute as $|\vec x| \rightarrow \infty$, in the sense that for
$N_j \in \mathcal S_0$
 \[
 \sup_{I \subset \mathbb R/(\frac{\pi}{2} \mathbb Z)} \bigg|\varrho_{N_j}(I)-\frac{|I|}{\pi/2}\bigg| \ll \frac{1}{N_j^{\log (\pi/2)/\log 2-\varepsilon}}
 \]
 where $\log (\pi/2)/\log 2=0.651496129\ldots$, for any
 $\varepsilon>0$.  The restriction to a density one
 subsequence of $\mathcal S$ is best possible since in the case where
 $|\vec x|^2$ is a prime congruent to $1 \Mod 4$ there are only two
 angles $\tmod{\pi/2}$; in fact, the set of possible limiting measures
 is very rich even after restricting to circles having a growing number
 of lattice points on them (cf. \cite{attainable}.)  Interestingly, the result above shows at scales $N^{-\gamma}$ with $\tfrac12< \gamma <\log(\pi/2)/\log 2$ that
   \begin{figure}[h]
    \centering
    \includegraphics[width=1 \textwidth]{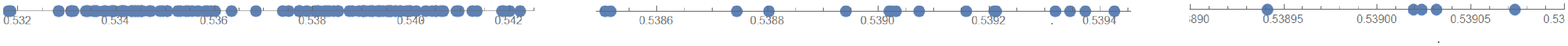}
    \caption{Three plots of angles of lattice points lying on the circle of radius $\sqrt{n}$ with $n=3.7366813\cdot 10^{35}$ and of $N=r(n)=2^{14}$ nearby $0.5390$  with windows of length $1/100,1/1000,1/5000$, respectively. Note that $N^{-1/2}=0.0078125\ldots$, $N^{-\log (\pi/2)/( \log 2)}=0.0017960\ldots$, and $N^{-1}=0.0000610\ldots$.
}
    \label{fig:mesh1}
\end{figure}
 lattice points on circles do not behave like independent identically distributed (iid) random variables, 
 uniformly distributed on $[0,1]$, $U_n$, $n=1,2,\ldots$, for which the Chung-Smirnov law of the iterated logarithm \cite{chung,smirnov} gives almost surely (a.s.) that
\[
\limsup_{N \rightarrow \infty} \frac{\sqrt{N}}{\sqrt{2 \log \log N}}  \sup_{I \subset \mathbb T} \big|\nu_N(I)-|I|\big| =1.
\]
  \begin{figure}[h]
    \centering
    \includegraphics[width=1 \textwidth]{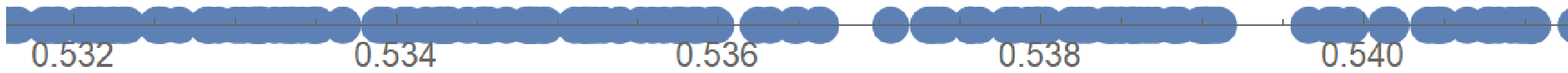}
    \caption{$2^{14}$ random points in $[0,1]$. Three plots of the points that lie nearby $0.5390$ with windows of length $1/100,1/1000,1/5000$, respectively.
}
    \label{fig:mesh2}
\end{figure}
While behaving differently at intermediate scales, at the local scale,
$1/N$, numerical evidence indicates that lattice points on circles
appear to typically display Poissonian statistics (cf. Figure
\ref{fig:mesh3} below), which would coincide with the known behavior
of $U_n$ at this scale (cf. \cite[\S 1.7]{feller-probability-vol2})).  
  \begin{figure}[h]
    \centering
    \includegraphics[width=1 \textwidth]{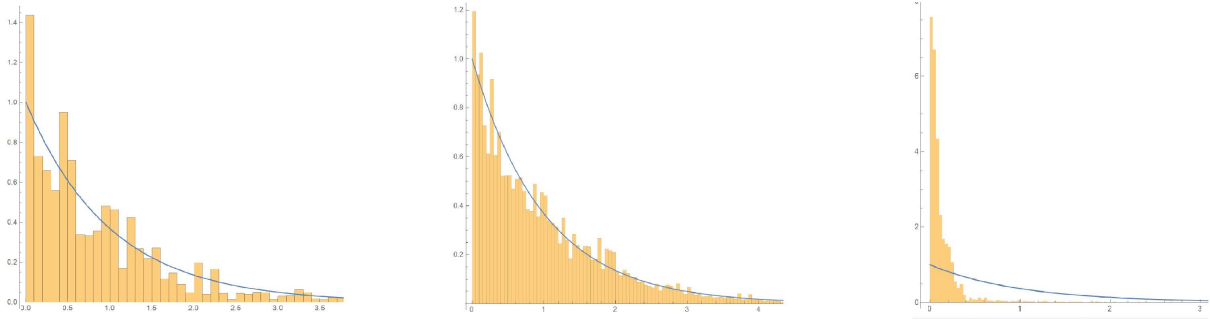}
    \caption{Histogram plots of the nearest neighbor spacing of $2^{14}$ angles of lattice points on the circle of radius $\sqrt{n}$ with $n=3.7366813\cdot 10^{35}$ (left), $2^{20}$ angles with $n=2.1957869 \cdot 10^{54}$ (center) and $2^{16}$ angles with $n=9.1943528\cdot 10^{63}$ (right), with plot of $e^{-s}$ (solid lines). The plot on the right is atypical; here $n$ is a product of primes of the form $m^2+1$.
}
    \label{fig:mesh3}
\end{figure}

We prove that the nearest neighbor spacing of angles of lattice points on circles is Poissonian along a density one subsequence of admissible radii. 
%
Listing the lattice points on the circle of radius $\sqrt{n}$ as $0 \le \theta_1 <\theta_2< \cdots < \theta_N < \pi/2$,
we define the 
nearest neighbor spacing measure on $\mathbb R_{\ge 0}$ by 
\begin{equation}
  \label{eq:spacing-measure-def}
\mu_{N}=\frac{1}{N-1}\sum_{1 \le j \le N-1} \delta(N(\theta_{j+1}-\theta_j) ).  
\end{equation}

\begin{thm} \label{thm:poisson} There exists a density one subsequence
  $\{ n_j \} \subset \mathcal S$ such that for any interval $I \subset 
  \mathbb R_{\ge 0}$ and $N_j=r(n_j)$ as $j \rightarrow \infty$ we
  have  
\[
\mu_{N_j}(I) = \int_{I} e^{-s} \, ds+o_{I}(1).
\]

\end{thm}

The restriction to a density one subsequence in our result is
essentially best possible, since there exist subsequences
$\{n_j\} \subset \mathcal S$, for which all the angles are highly
localized near $0$ or $\pi/2$. This remains true even if one requires
$r(n_j) \rightarrow \infty$ as $n_j \rightarrow \infty$ (cf. the plot
on the right in Figure \ref{fig:mesh3}.)  In fact, this construction
can easily be made into a rigorous argument showing that there exist
subsequences of elements of $\mathcal S$ for which the nearest
neighbor spacing measure is given by a delta mass supported at
zero.

Moreover, we are also able to compute the joint limiting
distribution. Consider the measure on $\mathbb R_{\ge 0}^{\ell}$ given
by 
\[
\mu_{\ell,N}=\frac{1}{N-\ell} \sum_{1 \le n \le N-\ell} \delta(N(\theta_{j+1}-\theta_j, \theta_{j+2}-\theta_{j+1}, \ldots, \theta_{j+\ell}-\theta_{j+\ell-1})).
\]

\begin{thm} \label{thm:poissonjoint} There exists a density one
  subsequence $\{ n_j \} \subset \mathcal S$ such that for
  any intervals
  $I_1,\ldots, I_{\ell} \subset \mathbb R_{\ge 0}$, where
  $\ell \in \mathbb N$ is fixed, and $N_j=r(n_j)$ as
  $j \rightarrow \infty$ we have
\[
\mu_{\ell,N_j} (I_1,\ldots, I_{\ell}) = \int_{I_1 \times \cdots \times
  I_{\ell}} e^{-(s_1+\cdots +s_{\ell})} \, ds_1 \cdots
ds_{\ell}+o_{I_{1},\ldots, I_{\ell}}(1).
\]
\end{thm}



We remark that the methods can be adapted to work for
``thinner'' subsequences inside $\mathcal S$, e.g. giving Poisson
spacings in a density one subsequence of square free integers $n \in S$
having exactly $\lfloor \tfrac12 \log \log n \rfloor$ prime factors.

\subsection{Discussion of the proof}
As shown in \cite{kurlberg-rudnick, katz-sarnak}, to show the nearest
neighbor spacings have Poissonian statistics it suffices to compute
the $r$-level correlations of the angles. As in previous works, we
will work with a smoothed version of the $r$-correlation function and
given a Schwartz function  $f:\mathbb
R^{r-1} \rightarrow \mathbb R$ we define 
\begin{equation} \label{eq:Fform}
\begin{split}
  F_N(x_1,\ldots,x_{r-1})=&
  \sum_{j_1,\ldots, j_{r-1} \in \mathbb Z}
  f\left( \frac{N}{\pi/2}\left(x_1+j_1 \cdot \frac{\pi}{2}, \ldots,
      x_{r-1} +j_{r-1} \cdot \frac{\pi}{2}\right) \right) \\
  =& \frac{1}{N^{r-1}}  \sum_{k_1,\ldots, k_{r-1} \in \mathbb Z}
  \widehat f\left(\frac{k_1}{N},\ldots, \frac{k_{r-1}}{N}\right)
  e^{4i(k_1x_1+\cdots+k_{r-1}x_{r-1})}
\end{split}
\end{equation}
where the second equality follows from applying the Poisson summation formula.
The $r$-correlation of the angles $0 \le \theta_1 < \theta_2 < \ldots
< \theta_N < \pi/2$ arising from lattice points on the circle of
radius $\sqrt{n}$ is defined by
\[
  R_r(n;F_N)
  =
  \frac{1}{N} \sum_{1 \le \ell_1,\ldots, \ell_r \le N } F_N(\theta_{\ell_1}-\theta_{\ell_2},\theta_{\ell_2}-\theta_{\ell_3},\ldots,\theta_{\ell_{r-1}}-\theta_{\ell_r}).
\]
We make no further assumption on $f$ and using a standard
combinatorial sieving argument we can relate the $r$-correlation over
\textit{distinct} angles to $r'$-correlations over all angles, as
above, with $r' \le r$.
To evaluate the $r$-correlation we will apply \eqref{eq:Fform} to see that
\begin{equation}\label{eq:rcorrdef-intro}
R_r(n;F_N)=\frac{1}{N^r} \sum_{\vec k \in \mathbb Z^{r-1}} \widehat
f\bigg(\frac{\vec k}{N} \bigg)
\lambda_{-4k_1}(n)\lambda_{4(k_1-k_2)}(n)\cdots
\lambda_{4(k_{r-2}-k_{r-1})}(n)\lambda_{4k_{r-1}}(n) 
\end{equation}
where $\lambda_{4k}(n)=\sum_{1 \le \ell \le N} e^{-i 4k\theta_{\ell}}$
is the $4k$th Fourier coefficient of the measure
$\sum_{1 \le \ell \le N} \delta(\theta_{\ell})$.
(Note that the set of lattice points on a circle is invariant under
rotation by $\pi/2$, hence the $k$th Fourier coefficient of the
measure is zero, unless $4 | k$. Further, this order four symmetry
implies that any $r$-tuple of distances occurs with multiplicity
divisible by four. Also note that the additional symmetry given by
complex conjugation will increase the multiplicities of pairwise
differences, but no such effect occurs for $r$-tuples of differences
for $r>2$.)

Our first step is to understand the average of $R_r(n;F_N)$ over $n$. To explain our approach, 
let us first consider the pair correlation $R_2(n;F_N)$. 
Our strategy allows us to freeze the value of $N$ and let $n$ vary
over all $n \le x$ with $r(n)=N$ for each $N \le (\log x)^{O(1)}$ (which is not very restrictive since $r(n)=(\log n)^{\frac{\log 2}{2}+ o(1)}$ for a density one subset of
$\mathcal S$.) 
At this point we use the Landau-Selberg-Delange method to evaluate the
mean value of $\lambda_{-4k}(n)\lambda_{4k}(n)=\lambda_{4k}(n)^2$ over
$n \le x$ with $r(n)=N$. Here we rely on the fact that
$L(s,4k)=\sum_{n \ge 1} \lambda_{4k}(n) n^{-s}$ is a Hecke
$L$-function and analogously to the Riemann zeta-function the function
$L(s,4k)$, has an Euler product, analytic continuation, and zero free
region.  
Noting that $\lambda_0(n)=r(n)=N$, arguing in this way we can show
that
\begin{equation} \label{eq:meanvalue}
\frac{1}{\#\{n \le x : r(n)=N\}}\sum_{\substack{n \le x \\ r(n)=N}} \lambda_{4k}(n)^2 = \begin{cases}
N^2 &\text{ if } k=0\\
N (g(k;Y)+\mathcal E_k) &  \text{ otherwise},
\end{cases}
\end{equation}
where there error term $\mathcal E_k$ is typically $o(1)$,
$Y=(\log N/\log 2-1)/\log \log x$ and $g(k;Y)$ is an Euler product
such that $g(k;Y)/L(1,8k)^Y$ is bounded.  The distinction between the
terms with $k=0,k \neq 0$ arises since the Dirichlet series
$ \sum_{n \ge 1} \lambda_{4k}(n)^2 n^{-s} $ has a pole of order $2$ at
$s=1$ if $k=0$, whereas if $k \neq 0$ it has a pole of order $1$ and
the residue will be given in terms of $L(1,8k)$. We note that the
total contributions are of similar order of magnitude since the order $2$
contributions occur on a subspace of larger codimension. At this point we
use \eqref{eq:meanvalue} in \eqref{eq:rcorrdef-intro} to get
\[
  \frac{1}{\# \{ n \le x: r(n)=N\}}\sum_{\substack{n \le x \\ r(n)=N}}
  R_2(n;F_N)=\widehat f(0)+\frac{1}{N}\sum_{k \in \mathbb Z \setminus \{0\}} \widehat f\bigg( \frac{k}{N} \bigg)  g(k;Y)+o(1).
\]
For $f$ that approximates the indicator function of an interval $I$ we
have $\widehat f(0)\approx |I|$, which is consistent with a sequence
that displays Poisson spacing statistics. To compute the second term
above we use analytic properties of $L(s,8k)$ to express $g(k;Y)$ in
terms of a {\em short} Dirichlet polynomial of length $\le N^{o(1)}$
with coefficients related to $\lambda_{4k}(n)$, that is, $g(k;Y)$ is
well-behaved as $k$ varies and we can apply Poisson summation again to
show the second term above equals $f(0)+o(1)$. To pass to 
the pair correlation function over distinct angles
we subtract the contribution from the terms with $\theta_{\ell_1}=\theta_{\ell_2}$, which equals $\tfrac{1}{N}\cdot N F_N(0)=f(0)+o(1)$ and cancels the secondary main term above.
To show that it is Poisson along a density one subsequence we will
bound the variance by computing the average of $R_2(n;F_N)^2$
(following a similar strategy), similar to the argument used by Sarnak
to show Poisson pair correlation for Laplace eigenvalues of generic
$2d$ tori \cite{sarnak-poisson-pair}.  An interesting obstruction to
extending this for $r>2$ in Sarnak's setting is ``variance blowup''
--- expectations remain consistent with Poisson statistics, but
variance growth makes it impossible to deduce Poisson behavior for
almost all tori.

Fortunately there is no such
variance blowup in our setting for larger $r$, but the approach
described above becomes significantly more involved and further ideas
are needed.
As before, our first step is to apply the Landau-Selberg-Delange
method to transform the problem, which gives a new expression that
depends on $\vec k \in \mathbb Z^{r-1}$ in a complex way including an arithmetic factor
$g(\vec k ;Y)$ as
well as an analytic factor which comes from the order of the pole of
the Dirichlet series associated to 
$ \lambda_{-4k_1}(n)\lambda_{4(k_1-k_2)}(n)\cdots
\lambda_{4(k_{r-2}-k_{r-1})}(n)\lambda_{4k_{r-1}}(n)$,
at
$s=1$.  The arithmetic
factor $g(\vec k;Y)$ is roughly of the shape of a product of Hecke
$L$-functions at $s=1$, which we can approximate by a short Dirichlet polynomial, so if we were to ignore the analytic factor we would be able 
to evaluate the average over $\vec k$.  While the behavior of the
analytic factor is complex, its value is fixed over certain subregions
of $\mathbb Z^{r-1}$ (e.g. for $r=2$, this corresponds to $k=0$,
$k \neq 0$) and by partitioning $\mathbb Z^{r-1}$ into these
subregions we are able to perform the average over $\vec k$ by means
of another application of the Poisson summation formula over each
nontrivial subregion, each of which is essentially a lattice.  This procedure essentially solves the arithmetic part of the problem. 

However, the procedure described above yields a term for each
subregion in the decomposition of $ \mathbb Z^{r-1}$ that gives a
contribution that is potentially the same size as our main term or
potentially even
larger (later on we are able to rule out any terms ``blowing up'' as $N
\rightarrow \infty$.) We need to control the terms arising from these
subregions of $\mathbb Z^{r-1}$ and would like to understand the
behavior of the analytic factor, which counts the order of the pole at
$s=1$ of the 
Dirichlet series mentioned above. It turns out, this is $\ge
1$ for every $\vec k \in \mathbb Z^{r-1}$, so that there are poles everywhere! 
Moreover, the number of terms arising in the decomposition of $\mathbb Z^{r-1}$ into subregions where the value of the analytic factor is fixed grows exponentially with $r$ making it difficult to  analyze for large $r$. 
This complexity arises in part due to residual lower order correlations and we
are left with an explicit yet intractable combinatorial expression for
the $r$-correlation.

A similar predicament also occurs in the context of
computing the $r$-correlation of the zeros of the Riemann
zeta-function, where after solving the arithmetic part of the problem
 one is left with an unwieldy expression for
the $r$-correlation. In the range corresponding to the ``diagonal terms''
Rudnick and Sarnak \cite{Rudnick-Sarnak0,Rudnick-Sarnak} succeeded in
solving the 
combinatorial problem. The ``off-diagonal'' terms can be analyzed at
the heuristic level using Hardy-Littlewood's conjecture on
correlations of primes (e.g. \cite{montgomery}), however for large
$r$, this leaves a horrible combinatorial formula. By expanding the
random matrix $r$-correlation function as a sum of cycles, Bogomolny
and Keating \cite{bk1,bk2} were able to eventually arrive at the same
expression (including the diagonal terms), thereby heuristically
matching the correlations of zeros 
of $\zeta(s)$ with that of eigenvalues
of GUE random matrices in the full range.  
Innovative recent works
\cite{entin-roditty-gershon-rudnick, conrey-snaith, chandee-lee,
  mason-snaith} on statistics of zeros of $L$-functions have solved
closely related 
combinatorial problems through relating these expressions to identical
ones coming from $L$-functions over function fields
\cite{entin-roditty-gershon-rudnick} or from random matrices
\cite{conrey-snaith, chandee-lee, mason-snaith}, for which the spacing
statistics of the zeros/eigenvalues, respectively, are known (see also \cite{conrey-keating}). We
pursue an approach similar in spirit by
%
%
%
%
introducing a random model for the angles of lattice points ---
interesting on its own right --- and compute the spacing statistics
for this model in two ways. First, we argue directly\fixmepk{add more
  comments on the proof. In particular mention lower order terms that
  do not show up in ``take $N$ random uniform points''-model } and
show that the $r$-level correlations of the random model are
Poissonian. Second, we then compute the $r$-level correlation in the
random case in a different way
and eventually arrive at the same complicated expression as in the
deterministic setting, including the exact same analytic
factor. Hence, we are able to match the deterministic computation with
the random one and then are able to evaluate the $r$-correlation by
combining these two approaches.

\subsection{Acknowledgments}

P.K. was partially supported by the Swedish Research Council (grant
numbers 2016-03701,2020-04036). S.L. was partially supported by an
EPSRC Standard Grant EP/T028343/1.  We would like to thank Jon Keating
and Peter Sarnak for comments on an earlier version of this
manuscript.

\section{The set-up}

Let $\mathcal N_M$ be the set of square free integers with prime divisors all congruent to $1 \tmod 4$ and total number of prime divisors equal to $M$ i.e.
\[
\mathcal N_M= \{ n \in \mathbb N : \mu^2(2n)b(n)=1 \text{ and } \Omega_1(n)=M\},
\]
where $b(\cdot)$ is the indicator function of the set of integers which can be represented as a sum of two squares
and 
\[
\Omega_1(n) =\sum_{\substack{p^a || n \\ p \equiv 1 \pamod4}}a.
\]
Throughout we use the notation $1_P$, which we define as $1_P=1$ if the statement $P$ is true and $1_P=0$ otherwise. We also write $\one_{\mathcal A}$ for the indicator function of the set $\mathcal A$ e.g. $\one_{\mathcal S}=b$.
Note that for each $n \in \mathcal N_M$ we have $r(n)=2^M$. 
Also let
\[
\mathcal N_M(x)=  \{ n \in \mathcal N_M : n \le x\}.
\]
%
%
Let $\OO=\mathbb Z[i]$ denote the Gaussian integers.
For $n$ with $r(n)=2^M=N$ we express the $r$-correlation of the lattice
points on the circle with radius $\sqrt{n}$ by
\[
R_{r}(n;F_N) =\frac{1}{N} \sum_{
\substack{
(\beta_1), \ldots, (\beta_{r})  \subset \OO \\
|\mathcal O/(\beta_1)|=\ldots= |\mathcal O/(\beta_r)|=n     
}}
F_N(\theta_{\beta_1}-\theta_{\beta_2}, \ldots,
\theta_{\beta_{r-1}}-\theta_{\beta_r}), 
\]
with $F_{N}$ as defined in  \eqref{eq:Fform},
and $\theta_{\beta}=\tmop{arg}(\beta)$ --- note that the angle
$\arg(\beta)$ is well defined modulo $\pi/2$ since it is 
independent of choice of generator of $(\beta)$.
We recall that the reason for working with period $\pi/2$,
rather than $2 \pi$, is that the set of lattice points on a given
circle is invariant under rotation by $\pi/2$.

Because we will restrict to a density one set of integers which are sums of
two squares, it suffices to consider sums of two squares $n=ef^2$ with
$e$ square free and $f \le F$ for $F$ 
which grows arbitrarily slowly as $x \rightarrow \infty$ since the
number of such sums of two squares $ \le x$ with $f \ge F$ is
\begin{equation} \label{eq:fixedn0}
\sum_{f \ge F} \sum_{e \le x/f^2} b(ef^2) \mu^2(e) \ll \frac{x}{F \sqrt{\log x}}.
\end{equation}
Hence, it suffices to consider $n=n_0n_1$ where $n_0=2^af^2d$,
$1 \le f \le F$, $a \in \{0,1\}$, $d$ is a divisor of $f$ with
$\mu^2(2d)b(d)=1$ and $n_1 \in \mathcal N_M$ with $(n_0,n_1)=1$ and
$1 \le M \le A
\log \log x$ for any fixed $A>0$.
It is then sufficient to average over \textit{odd} square free
  numbers $n_1$; this partitioning simplifies some of
  the analysis 
  below but  is not essential to our argument.  Write
$\mathcal N_{M,n_0}=\{ n \in \mathcal N_M: (n,n_0)=1 \}$ and
$\mathcal N_{M,n_0}(x)=\{ n \in \mathcal N_{M,n_0}: n \le x\}$. We
also define
\[
R_{r,n_0}(n;F_N):=R_r(n_0n;F_N).
\]

\section{A random model}
\label{sec:random-model}

Motivated by Hecke's result on the equidistribution of angles of
Gaussian primes, we introduce a simple, purely probabilistic, model
for the angular gaps between lattice points $(x,y) \in \Z^2$ on the
circle $x^{2}+y^{2} = n$, for $n$ square free. Namely, for each prime
$p|n$, pick a random uniformly distributed angle
$\vartheta_{p} \in \T := \R/( (\pi/2) \Z)$ (which should not be confused
with the deterministic angle $\theta_p$ associated with $p$.)  For
such $n$, with probability one, we obtain $r(n)=2^{M}$ distinct points
on the unit circle, where $M$ is the number of prime divisors of $n$,
with angles given by
$$
\vartheta = \sum_{p|n} \pm \vartheta_p
$$

It is convenient to parameterize these angles as follows: order the set
of prime divisors $p|n$ according to size, i.e.,
$p_{1} < p_{2} < \ldots < p_{M}$ and (abusing notation) put
$\vartheta_i = 2 \vartheta_{p_{i}}$ for $i = 1, \ldots, M$ (note that
$\vartheta_1, \ldots, \vartheta_{M}$ are then independent and identically
distributed uniform random variables on $\T$), 
and for each subset
$J \subset \{1, \ldots, M\}$ define a random variable 
\begin{equation} \label{eq:xJ-def}
x_J := \sum_{j \in J } \vartheta_{j} \pmod{\pi/2};
\end{equation}
note that the gaps between these points and the original points differ
by a translation by $-\sum_{p|n} \vartheta_p $ and hence their spacings
as well as correlations are identical.

 We define the $r$-correlation of the points 
 $(\theta_{\beta} + 
x_J)_{\substack{ (\beta) \subset \OO :|\OO/\beta|=n_{0} \\ J\subset\{1,\ldots,M\}}}$
 by
\begin{equation}\label{eq:randomcorr}
\mathcal  R_{r,n_0}(F_N) =
  \frac{1}{N} \sum_{\substack{J_1,\ldots,J_r \subset \{1,\ldots,M\} \\ |\mathcal
      O/(\beta_1)|=\ldots= |\mathcal O/(\beta_r)|=n_0} }
  F_N(\theta_{\beta_1}-\theta_{\beta_2}+x_{J_1}-x_{J_2}, \ldots, \theta_{\beta_{r-1}}-\theta_{\beta_r}+ x_{J_{r-1}}-x_{J_r}),
\end{equation}
where $N=2^M r(n_0)$.  We then show that analogues of the mean and the
variance in the deterministic model are well matched by the mean and
variance of the random model.
\begin{thm} \label{thm:corr2}
Let $A>0$ be fixed.
Suppose that $1 \le M \le A\log \log x$.
We have that
\begin{equation}\label{eq:randomeq1}
\begin{split}
  \frac{1}{\# \mathcal N_{M,n_0}(x)}
  \sum_{n \in \mathcal N_{M,n_0}(x)} R_{r,n_0}(n;F_N)
  =& \mathbb E \big(  \mathcal R_{r,n_0}(F_N) \big)+O\left( \frac{1}{(\log \log x)}
+\frac{1}{M^{10}}\right).
\end{split}
\end{equation}
Additionally, 
\begin{equation}\label{eq:randomeq2}
\begin{split}
  \frac{1}{\# \mathcal N_{M,n_0}(x)}
  \sum_{n \in \mathcal N_{M,n_0}(x)} R_{r,n_0}(n;F_N)^2 
  =& \mathbb E \big(  \mathcal R_{r,n_0}(F_N)^2 \big)+O\left( \frac{1}{(\log \log x)}+\frac{1}{M^{10}}\right).
\end{split}
\end{equation}
The implied constants depend at most on $f,r,n_0$ and $A$.
\end{thm}
Note that
$r(n) = (\log n)^{(\log 2)/2+o(1)}$ holds for a full density
subset of the set integers for which $r(n)>0$, hence the condition
$M \in [1, A \log \log x]$ is not very restrictive.

\section{Applying the LSD method}
Let $k_1,\ldots, k_{r-1} \in \mathbb Z$. Here and in what follows we use
the convention that $k_0=k_r=0$. To compute the $r$-correlation we apply
\eqref{eq:Fform}. Switching the order of summation we will need to
understand the sum
\begin{equation} \label{eq:target}
\frac{1}{\# \mathcal N_{M,n_0}(x)}\sum_{n \in \mathcal N_{M,n_0}(x)} \prod_{j=0}^{r-1} \lambda_{4(k_{j+1}-k_j)}(n)\quad \text{ where } \quad
\lambda_{4k}(n)=\sum_{|\mathcal O/(\beta)|=n} e^{i4k\theta_{\beta}}
\end{equation}
for each given $k_1,\ldots, k_{r-1}$ 
(note that while $\beta$ is only defined up to multiplication by $i$,
the angle $4k\theta_{\beta}$ is well defined modulo $2 \pi$.)  This
will be done by applying the Landau-Selberg-Delange method, which
requires some analytic input.

Let us introduce the following Hecke $L$-functions.
Recall that $\OO=\Z[i]$ denotes the ring of integers of $\Q(i)$;
recall also that 
$\OO$ is a principal ideal domain.
Further,
for
$k \in \mathbb Z$ and an ideal
$\mathfrak b=(\beta)\subset \OO$, let
$\Xi_{4k}(\mathfrak b)=e^{4ki\theta_{\beta}}$ and define
(note that $|\OO/\mathfrak b|$ is the norm of the ideal $\mathfrak b$)
\begin{equation} \label{eq:lfn}
  L(s,4k)=
  \sum_{0 \neq \mathfrak b \subset \mathcal O}
  \frac{\Xi_{4k}(\mathfrak b)}{|\OO/\mathfrak b|^s}  
  =
  \prod_{\mathfrak p \subset \mathcal O}
  \left( 1-\frac{\Xi_{4k}(\mathfrak p)}{|\OO/ \mathfrak p|^s}\right)^{-1}
  =
  \prod_{p}\left(1-\frac{\lambda_{4k}(p)}{p^s}+\frac{\psi_4(p)}{p^{2s}} \right)^{-1}
\end{equation}
where the sum is over all nonzero ideals of $\OO$, the first product is over the set of prime ideals in $\OO$, and $\psi_4(p)$ is the non-principal Dirichlet character modulo $4$. The last identity can be seen by separately considering split, ramified, and inert primes.
The $L$-function above admits an analytic continuation to all of
$\mathbb C$ provided that $k \neq 0$; for $k=0$,
$L(s,0)=\zeta_{\mathbb Q(i)}(s)=\zeta(s)L(s, \psi_4)$, which has a
simple pole at $s=1$. Additionally, for primes $p \equiv 1 \pmod 4$ we
write $\theta_p=\theta_{(\pi)}=\tmop{arg}(\pi)$ where $\pi$ is the Gaussian prime with 
$|\pi|^2=p$ and $0<\tmop{arg}(\pi) < \pi/4$.

 Given $\mathbf S \subset \{0,1,\ldots, r-1\} $ and $k_{1}, \ldots,
 k_{r-1} \in \Z$ define 
 $$
 k_{\mathbf S}=k_{\mathbf S}(k_1,k_2,\ldots,k_{r-1})=
 \sum_{j \in  \mathbf S} (k_{j+1}-k_j) \in \Z
 $$
 and observe $k_{\mathbf S}+k_{\mathbf  S^c}=\sum_{j=0}^{r-1}
 (k_{j+1}-k_j)=0$ (recall $k_0=k_r=0$ by 
convention.) Consequently for $p \equiv 1 \tmod 4$,  we have that
$\lambda_{4k_{\mathbf S}}(p)=e^{i4k_{\mathbf S}
  \theta_p}+e^{-i4k_{\mathbf S} \theta_p}=e^{i4k_{\mathbf S}
  \theta_p}+e^{i4k_{{\mathbf S}^c} \theta_p}$. Motivated by the previous observation, let $\mathcal
M_r=\mathcal P(\{0,\ldots, r-1\})/ \sim$ where for $\mathbf S,\mathbf
T \in \mathcal P( \{0,\ldots, r-1\})$ we write $\mathbf S \sim \mathbf T$ if
$\mathbf S=\mathbf T^c$ where 
  $\mathbf T^c=\{0,\ldots, r-1\} \setminus \mathbf T$ and for a set $\mathcal A$ we use the notation $\mathcal P(\mathcal A)$ to denote the power set of $\mathcal A$. 
  Using the statements above we find that
\begin{equation}\label{eq:combin}
\begin{split}
\prod_{j=0}^{r-1} \lambda_{4(k_{j+1}-k_j)}(p)&=\prod_{j=0}^{r-1} e^{-4(k_{j+1}-k_j)i \theta_p}(1+e^{8(k_{j+1}-k_j)i\theta_p})\\
&=\prod_{j=0}^{r-1} (1+e^{8(k_{j+1}-k_j) i \theta_p})=\sum_{\mathbf S \subset \{0,1,\ldots, r-1\}} e^{8k_{\mathbf S}i\theta_p}=\sum_{ \substack{[\mathbf S] \in \mathcal M_r}} \lambda_{8  k_{\mathbf S}}(p).
\end{split}
\end{equation}
With $\vec{ k} = (k_{1}, \ldots, k_{r-1}) \in \Z^{r-1}$ define
\begin{equation} \label{eq:alphadef}
  \alpha=\alpha(\vec k)=
  \sum_{ \substack{[\mathbf S] \in \mathcal M_r \\k_{\mathbf S}=0}} 1,
\end{equation}
and note that $\# \mathcal M_r \le 2^{r-1} $.
Further, as
$k_{\emptyset}(\vec k)=0$ there is at least one term in the sum in
(\ref{eq:alphadef}), and thus we have
$1 \le \alpha(\vec k) \le 2^{r-1}$ for all
$ \vec k \in \mathbb Z^{r-1}$.
Using a classical
result of Hecke (see \cite[p. 130 \& Eq'n (5.52)]{IK}), which gives
equidistribution of angles of Gaussian primes,
we have that 
\begin{equation} \label{eq:pnt}
\sum_{\substack{p \le x \\ p \nmid n_0}}  \bigg( \prod_{j=0}^{r-1} \lambda_{4(k_{j+1}-k_j)}(p) \bigg)\log p=\alpha(\vec k) x +O(x/(\log x)^B)
\end{equation}
uniformly for $\max_j |k_j| \le e^{(\log x)^{1/3}}$ and any $B>0$,
where the implied
constant depends at most on $B,r$ and $n_0$.\footnote{Hence, it can be shown that the order of the pole of the Dirichlet series associated to $\prod_{j=0}^{r-1} \lambda_{4(k_{j+1}-k_j)}(n)$ at $s=1$ equals $\alpha(\vec k)$.}

We now state the main result of this section, which provides an estimate for \eqref{eq:target}.

\begin{prop} \label{prop:LSD}
Let
$Y=(M-1)/\log \log x$ and $A >0$ be fixed.
Let $\alpha(\vec k)$ be as in
\eqref{eq:alphadef}. Suppose
 $1 \le M \le A \log \log x$. Then there exist absolute constants $A',c_0>0$  such that for $\max_{[\mathbf S] \in \mathcal M_r}
|k_{\mathbf S}| \le e^{ (\log x)^{c_0}}$ we have that
\[
\frac{1}{\# \mathcal N_{M,n_0}(x) }\sum_{n \in \mathcal N_{M,n_0}(x)}
\prod_{j=0}^{r-1} \lambda_{4(k_{j+1}-k_j)}(n) = (2\alpha(\vec k))^M
\bigg(g(\vec k;Y)+O\bigg( \frac{\mathcal L(\vec k)}{\log \log x}
\bigg)\bigg), 
\]
where 
\begin{equation} \label{eq:glkdef}
  g(\vec k;Y)=
  \prod_{\substack{p \equiv 1 \Mod 4 \\ (p,n_0)=1}} \left(1+\frac{Y
      \sum_{\substack{[\mathbf S] \in \mathcal M_r \\ k_{\mathbf S}
          \neq 0}} \lambda_{8k_{\mathbf S}}(p)}{\alpha(\vec k) p
      (1+\frac{2Y}{p})} \right) \quad \text{and} \quad \mathcal L(\vec
  k)
  =\sum_{\pm} \sum_{\substack{[\mathbf S] \in \mathcal M_r \\ k_{\mathbf S} \neq 0}} |L(1,8k_{\mathbf S})|^{ \pm A'}. 
\end{equation}
\end{prop}

\subsection{Sums of multiplicative functions}
To prove Proposition \ref{prop:LSD} we will use general results for sums of multiplicative functions that are derived from the Landau-Selberg-Delange method.
Let $f$ be a multiplicative function such that 
\begin{equation} \label{eq:pnt2}
|f(n)| \le \sum_{d_1\cdots d_l=n}1 \quad  \text{ and } \quad
\sum_{p \le x}f(p) \log p=\gamma x+O(x/(\log x)^B)
\end{equation}
for some  $l \in \mathbb N$,$\gamma \in \mathbb C$, and $B
>0$.
Let 
\[
F(s)=\sum_{n \ge 1}\frac{f(n)}{n^s}.
\]
For each nonnegative integer $j$, define  
\begin{equation} \label{eq:cjdef}
\widetilde c_j(\gamma)=\frac{1}{j!} \cdot \frac{d^j}{ds^j} \bigg|_{s=1} \frac{(s-1)^{\gamma}F(s)}{s}.
\end{equation}
In particular, note that
\begin{equation} \label{eq:cform}
\widetilde c_0(\gamma)=\prod_p \left(1+\frac{f(p)}{p}+\frac{f(p^2)}{p^2}+\cdots \right)\left(1-\frac{1}{p} \right)^{\gamma}.
\end{equation}
We now state a result of Granville-Koukoulopoulos
\cite[Th'm~1]{Granville-Koukoulopoulos} 
on sums of
multiplicative functions.
\begin{thm}[Granville-Koukoulopoulos
\cite{Granville-Koukoulopoulos}] \label{thm:gk}
Suppose that \eqref{eq:pnt2} holds. Then for $J=\lfloor B \rfloor$
\[
\sum_{n \le x} f(n) = x \sum_{j=0}^J \widetilde c_j(\gamma) \frac{(\log x)^{\gamma-j-1}}{\Gamma(\gamma-j)}+O\left(x(\log x)^{l-B-1+o(1)}\right),
\]
where the implied constant depends at most on $B,l$ and the implied constant in \eqref{eq:pnt2}.
\end{thm}
We also require the following result, which is a special case of
\cite[Th'm II.6.1.3]{Tenenbaum}.  

\begin{thm}[Tenenbaum \cite{Tenenbaum}] \label{thm:Tenenbaum}
Let $z\in \mathbb C$.
Let $a_z(\cdot)$ be an arithmetic function for which there exists $A>0$ such that for each $n \in \mathbb N$ there is a power series expansion in the disc $|z| \le A$ of the form
\[
a_z(n)=\sum_{M=0}^{\infty} c_M(n) z^M.
\]
Suppose there exists a function $h(z)$ that is holomorphic in $|z| \le A$ such that for $x \ge 3$ and $|z| \le A$
\[
\sum_{n \le x} a_z(n)=x(\log x)^{z-1}(zh(z)+O(R(x)))
\]
for some quantity $R(x)$, which does not depend on $z$ and the implied constant depends at most on $A$. 
Then for $x \ge 3$ and $1 \le M \le A \log \log x$ we have
\[
\sum_{n \le x} c_M(n)= \frac{x}{\log x} \frac{(\log \log x)^{M-1}}{(M-1)!}\left(h\left(\frac{M-1}{\log \log x}\right) +O\left( \frac{(M-1)}{(\log \log x)^2} \max_{|z| \le A} |h''(z)| +\frac{\log \log x}{M} R(x) \right)\right) 
\]
where the implied constant depends at most on $A$.
\end{thm}

\subsection{Analytic estimates}
Using the generalized binomial theorem, we see that for $z \in \mathbb
C$ and $\tmop{Re}(s)>1$ that 
\[
  L(s,4k)^z=
  \prod_{\mathfrak p \subset \mathcal \OO} \left(
  \sum_{j=0}^{\infty} \binom{z+j-1}{j}
  \frac{\Xi_{4k}(\mathfrak p^j)}{|\OO/\mathfrak p^j|^s}
  \right)
  =\sum_{n \ge 1}
  \frac{\lambda_{4k}(n;z)}{n^s}
\]
where $\binom{\cdot}{\cdot}$ denotes the generalized binomial
coefficients (in particular, we have $(1-t)^{z} = \sum_{k=0}^{\infty}
\binom{z+k-1}{k} t^{k}$ for $|t|<1$)
and
$\lambda_{4k}(\cdot;z)$ is the multiplicative function 
given by 
\begin{equation} \label{eq:lzdef}
\lambda_{4k}(n;z)= \sum_{ \mathfrak a \subset \OO :  |\mathcal O/\mathfrak a|=n} b_z(\mathfrak a)
\Xi_{4k}(\mathfrak a), 
\end{equation}
where the sum is over ideals in $\OO$ having norm $n$, and
for $\mathfrak a=\mathfrak p_1^{a_1}\cdots \mathfrak p_w^{a_w}$,
we let $b_z(\mathfrak a)=\binom{z+a_1-1}{a_{1}} \cdots
\binom{z+a_w-1}{a_{w}}$. We 
have for $|z| \le A$ that 
\begin{equation} \label{eq:lzbds}
b_z(\mathfrak a) \ll |\mathcal O/\mathfrak a|^{o(1)}, \qquad \lambda_{4k}(n;z) \ll n^{o(1)},
\end{equation}
uniformly for $k \in \mathbb Z$.
For $Y \ge 0$, $\vec k \in \mathbb Z^{r-1}$ define the multiplicative functions $w(\cdot;Y),s(\cdot;\vec k)$
supported on square free numbers given by
\begin{equation} \label{eq:wsdef}
w(p;Y)=\frac{Y b(p)}{ 1+\frac{2Y}{p}}, \qquad s(p;\vec
k)=\sum_{\substack{[\mathbf S] \in \mathcal M_r \\ k_{\mathbf S} \neq
    0}} \lambda_{8k_{\mathbf S}}(p). 
\end{equation}

The following key proposition will allows us to
express $g(\vec k;Y)$ (cf. \eqref{eq:glkdef} for its definition) as a 
short Dirichlet polynomial (later we will choose
$y=\exp( (\log \lVert \vec k \rVert_{\infty})^{3/4})$.) 
\begin{prop} \label{prop:approx}
Let $K=\max_{[\mathbf S] \in \mathcal M_r} |k_{\mathbf S}|$ and $A >0$. Then
for $y \ge \exp( (\log K)^{2/3+o(1)})$ and $ 0 \le Y \le A$, we have that
\begin{equation} \label{eq:gform}
g(\vec k;Y)=\sum_{\substack{m \le  y \\(m,2n_0)=1}} \frac{\alpha(\vec k)^{-\Omega_1(m)}w(m;Y) s(m;\vec k)}{m}+O\left( \exp\left(\frac{- \log y}{(\log(K+y))^{2/3+o(1)}} \right)\right),
\end{equation}
where the implied constant depends at most on $n_0,A$.
Additionally, for $k \neq 0$, $y \ge \exp( (\log k)^{2/3+o(1)})$, and $ |\tmop{Re}(z)| \le  \frac{\log y}{(\log(k+y))^{2/3+o(1)}}$ we have
\begin{equation} \label{eq:lzform}
L(1,4k)^z=\sum_{m \le y} \frac{\lambda_{4k}(m;z)}{m}+O\left( \exp\left(\frac{- \log y}{(\log(k+y))^{2/3+o(1)}} \right)\right).
\end{equation}
\end{prop}

Before proceeding to the proof let us note that for each $1 \le j \le r-1$ choosing
 $\mathbf S=\{0,1,\ldots, j-1\}$ yields $k_{\mathbf S}=k_j$, so that $\lVert \vec k \rVert_{\infty} \le K$.


\begin{proof}
Since \eqref{eq:gform} and \eqref{eq:lzform} follow from similar
arguments we will only prove \eqref{eq:gform}. 
For $\tmop{Re}(s)>1$ define $Z(s) = Z(s, \vec k)$ by
\[
  Z(s) =
  \prod_{(p,2n_0)=1}\left(1+\frac{Y b(p)\sum_{\substack{[\mathbf S] \in \mathcal M_r \\
          k_{\mathbf S} \neq 0}} \lambda_{8k_{\mathbf S}}(p)}{\alpha(\vec k)
      (1+\frac{2Y}{p})p^{s}} \right)
  =
  \sum_{\substack{m \ge 1 \\(m,2n_0)=1}} \frac{\alpha(\vec k)^{-\Omega_1(m)}w(m;Y)s(m;\vec k)}{m^{s}}.
\]
Further, for $\tmop{Re}(s)>1$,
define $\mathcal G(s) = \mathcal G(s, \vec k)$ by
\[
\mathcal G(s)=Z(s)\prod_{\substack{[\mathbf S] \in \mathcal M_r \\ k_{\mathbf S} \neq 0}}L(s,8k_{\mathbf S})^{-Y/\alpha(\vec k)}.
\]
By using the Euler product representations of $Z(s)$ and $L(s,8k)$
it follows that $\mathcal G(s)$ is analytic in $\tmop{Re}(s) \ge 1/2+\varepsilon$
and in this region $|\mathcal G(s)| \ll 1$ where the
implied constant 
depends at most on $A,n_0,r,\varepsilon$ (recall that $1 \le \alpha(\vec k) \le 2^{r-1}$.) Let us state the Vinogradov-Korobov bounds  
\begin{equation} \label{eq:cole}
|L(s, 8k)|^{\pm 1} \ll (\log (|k|+|s|))^{2/3+o(1)}
\end{equation}
due to Coleman \cite[Thm. 1, 2]{Coleman},
which are valid in the region $s=\sigma+it$ with
\[
\sigma \ge 1- \frac{1}{(\log (|k|+|t|))^{2/3+o(1)}}
\]
(see also \cite[Lem. 8]{Coleman} and \cite[Thm. 3.10-3.11]{Titchmarsh}, cf. \cite[Lem. 10,11]{bingrong}.)

Hence, the result follows using a standard contour integration
argument by applying Perron's formula as follows. Let $T=y^3+K$. We
get that 
\[
  \sum_{\substack{m \le y \\ (m,2n_0)=1}} \frac{\alpha(\vec k)^{-\Omega_1(m)}w(m;Y) s(m;\vec k)}{m}
  =
\frac{1}{2\pi i} \int_{1-iT}^{1+iT} Z(s+1) \frac{y^s}{s} \, ds+O(y^{-1/2}).  
\]
Shifting contours to $\sigma=-1/(\log T)^{2/3+o(1)}$ we pick up a
simple pole at $s=0$ with residue $Z(1)=g(\vec k;Y)$. The horizontal
contour integrals are easily seen to be $\ll y^{-2} (\log
T)^{O(1+Y)}$, using \eqref{eq:cole}. The vertical contour on the line
$\sigma=-1/(\log T)^{2/3+o(1)}$ is  
\[
\ll (\log T)^{O(1+Y)} \exp \left( \frac{- \log y}{3(\log (K+y))^{2/3+o(1)}}\right)
\]
also using \eqref{eq:cole}. Collecting estimates completes the proof.
\end{proof}

\subsection{Applications of the LSD method}
To prove Proposition \ref{prop:LSD} we require estimates for $\# \mathcal N_{M,n_0}(x)$ and $\sum_{n \in \mathcal N_{M,n_0}(x)} \prod_{j=0}^{r-1} \lambda_{4(k_{j+1}-k_j)}(n)$. We will only estimate the latter quantity since the argument for estimating $\#\mathcal N_{M,n_0}(x)$ is similar, yet simpler.

For $w \in \mathbb C$ consider the multiplicative function supported
on integers co-prime to $n_0$ given by
\begin{equation} \label{eq:fdef}
f(n;w)= w^{\Omega(n)}b(n) \mu^2(2n) \prod_{j=0}^{r-1} \lambda_{4(k_{j+1}-k_j)}(n),
\end{equation}
where if $w=0$ we set $f(1;0)=1$,
and let
\begin{equation} \label{eq:Fdef}
F(s;w)= \sum_{n \ge 1} \frac{f(n;w)}{n^s}= \prod_{(p,2n_0)=1} \left(1+\frac{w b(p) \sum_{[\mathbf S] \in \mathcal M_r} \lambda_{8k_{\mathbf S}}(p)}{p^s} \right),
\end{equation}
where in the last line we have also used \eqref{eq:combin}.
Using the expression above we see that we can write
\begin{equation} \label{eq:ep2}
F(s;w)=H_{n_0}(s;w)\prod_{[\mathbf S] \in \mathcal M_r} L(s,8k_{\mathbf S})^w = H_{n_0}(s;w)\zeta(s)^{\alpha(\vec k) w} L(s,\psi_4)^{\alpha(\vec k) w} \prod_{\substack{[\mathbf S] \in \mathcal M_r \\ k_{\mathbf S} \neq 0}} L(s,8k_{\mathbf S})^w,
\end{equation}
where $H_{n_0}(s;w)$ is an analytic function with $|H_{n_0}(s;w)| \ll 1$ for $\tmop{Re}(s)\ge \tfrac12+\varepsilon$
  and $|w|\le A$, where the implied constant depends at most on $A,\varepsilon,n_0,r$ (recall $1 \le \alpha(\vec k) \le 2^{r-1}$.) 
Additionally, by \eqref{eq:pnt} for any $B>0$ and $|w|\le A$
\begin{equation} \label{eq:fpnt}
\sum_{p \le x} f(p;w) \log p=\alpha(\vec k) w x+O\left( \frac{x}{(\log x)^B}\right),
\end{equation}
uniformly for $K:=\max_{[\mathbf S] \in \mathcal M_r} |k_{\mathbf S}| \le e^{(\log x)^{1/3}}$ and the implied constant depends on $A,B,n_0$ and $r$.
For $F(s)=F(s;w)$ let $\widetilde c_j(\gamma)$ be as in
\eqref{eq:cjdef} (so that $\gamma=\alpha(\vec k) w$.) For
$|s-1| \le 1$ we can write $\zeta(s)=G(s)/(s-1)$ where $G(s)$ is a
nonzero analytic function in $|s-1|\le 1$ with $G(1)=1$. Hence, we have that 
\[
\frac{d^j}{ds^j} \bigg|_{s=1} \frac{(s-1)^{\alpha(\vec k) w} \zeta(s)^{\alpha(\vec k)
    w}}{s}
\ll 1
\]
for $|w|\le A$,
where the implied constant depends at most on $r,A, j$.
By this, \eqref{eq:ep2}, and \eqref{eq:cole}, say, it is easy
to see using Cauchy's theorem that
\begin{equation} \label{eq:ctildebd}
      |\widetilde c_j(\alpha(\vec k) w)| \ll (\log K)^{O(1)}
\end{equation}
for $|w| \le A$ where the implied constants depend at most on
$j,n_0,A,r$. 
 Applying Theorem \ref{thm:gk} and using \eqref{eq:ctildebd} gives that
\begin{equation} \label{eq:2sum}
\begin{split}
\sum_{n \le x} w^{\Omega(n)}b(n) \mu^2(2n) \prod_{j=0}^{r-1} \lambda_{4(k_{j+1}-k_j)}(n)= x (\log x)^{\alpha(\vec k) w-1}\left(\frac{\widetilde c_0(\alpha(\vec k) w)}{\Gamma(\alpha(\vec k) w)}+O\left( \frac{(\log K)^{O(1)}}{\log x}\right) \right)
\end{split}
\end{equation}
since $\alpha(\vec k) \neq 0$ where the implied constants depend at most on
$n_0,A,r$.

The next step of the proof is to
apply Theorem \ref{thm:Tenenbaum} and to match notation we write $z= \alpha(\vec k) w$ (recall that $1 \le \alpha(\vec k) \le 2^{r-1}$.)
Let $h(z)=\frac{\widetilde c_0(z)}{z 
  \Gamma(z)}$. For $|z|\le A$, \eqref{eq:2sum} implies
  \begin{equation} \label{eq:favg}
  \sum_{n \le x} f(n;z/\alpha(\vec k)) =x (\log x)^{z-1}\left(zh(z)+O\left( \frac{(\log K)^{O(1)}}{\log x}\right) \right).
  \end{equation}
  Since $K \le e^{(\log x)^{c_0}}$ and $c_0$ is sufficiently small the
  error term above is negligible.  By \eqref{eq:cform} (with $\gamma =
  \alpha(\vec{k})w=z$), \eqref{eq:ep2}
  and \eqref{eq:fpnt} we have that
\begin{equation} \label{eq:hform}
h(z)= \frac{1}{z\Gamma(z)} \prod_{p} \left( 1+\frac{f(p;\tfrac{z}{\alpha(\vec k)})}{p} \right)\left(1-\frac{1}{p}\right)^{z}=\frac{H_{n_0}(1;z/\alpha(\vec k)) L(1,\psi_4)^{z}}{z \Gamma(z)} \prod_{\substack{[\mathbf S] \in \mathcal M_r \\ k_{\mathbf S} \neq 0}} L(1,8k_{\mathbf S})^{z/\alpha(\vec k)}
\end{equation}
so that for $|z| \le A$ we have
\begin{equation} \label{eq:hderivbd}
|h''(z)| \ll \prod_{ \substack{[\mathbf S] \in \mathcal M_r \\ k_{\mathbf S} \neq 0}} \max_{\pm} |L(1,8k_{\mathbf S})|^{\pm (A+2)} \ll \sum_{\pm} \sum_{\substack{[\mathbf S] \in \mathcal M_r \\ k_{\mathbf S} \neq 0}} |L(1,8k_{\mathbf S})|^{ \pm A'}=: \mathcal L(\vec k)
\end{equation}
where the last step follows from the inequality of arithmetic and geometric means and $A'$ and the implied constant depend at most on $n_0,r,A$ (here we also used that $|\log t| \le \max(t,1/t)$ for $t>0$.) Taking  
$
a_z(n)=f(n;\tfrac{z}{\alpha(\vec k)})
$
we have for $|z| \le A$ that
\[
a_z(n)=\sum_{M=0}^{\infty} c_M(n) z^M, \qquad \text{where} \qquad 
c_M(n)=\begin{cases}
f(n;\alpha(\vec k)^{-1}) & \text{ if } M=\Omega(n), \\
0 & \text{ otherwise.}
\end{cases}
\]
Write $X=\log \log x$.
Hence, using \eqref{eq:favg}, \eqref{eq:hform}, and \eqref{eq:hderivbd} we may apply Theorem \ref{thm:Tenenbaum} to get that for $Y= \frac{M-1}{X}$
\begin{equation}\label{eq:lsd2}
\begin{split}
& \sum_{n \in \mathcal N_{M,n_0}(x)} \prod_{j=0}^{r-1} \lambda_{4(k_{j+1}-k_j)}(n)=\alpha(\vec k)^{M} \sum_{n \le x} c_M(n)\\
&=\alpha(\vec k)^{M} \frac{x}{\log x} \frac{X^{M-1}}{(M-1)!}\left( \frac{1}{Y\Gamma(Y)}\prod_{p} \left(1+\frac{f(p;\tfrac{Y}{\alpha(\vec k)})}{ p} \right) \left(1-\frac{1}{p} \right)^{Y}+O\left( \frac{M \mathcal L(\vec k)}{X^2}\right)\right).
\end{split}
\end{equation}
By a similar, yet simpler argument that we will omit we also get for $1 \le M \le A \log \log x$
that
\begin{equation}\label{eq:lsd1}
\# \mathcal N_{M,n_0}(x)=2^{-M} \frac{x}{\log x}
\frac{X^{M-1}}{(M-1)!}\left( \frac{1}{Y\Gamma(Y)}\prod_{p }
  \left(1+\frac{2Y 1_{(p,2n_0)=1} b(p) }{p} \right) \left(1-\frac{1}{p}
  \right)^{Y}+O\left( \frac{M}{X^2}\right)\right). 
\end{equation}
\subsection{Proof of Proposition \ref{prop:LSD}}

\begin{proof}[Proof of Proposition \ref{prop:LSD}]
The Euler product on the right-hand side of \eqref{eq:lsd1} is $\asymp 1$. Also, the ratio of the Euler product on the right-hand side of \eqref{eq:lsd2} to that on the right-hand side of \eqref{eq:lsd1} is
\[
\prod_{(p,2n_0)=1} \left(1+\frac{Y b(p) \sum_{\substack{[\mathbf S] \in \mathcal M_r} } \lambda_{8k_{\mathbf S}}(p)}{\alpha(\vec k) p} \right)\bigg/\left(1+\frac{2Y b(p) }{p} \right)=g(\vec k;Y),
\]
which follows from grouping the terms in the summation over $[\mathbf S] \in \mathcal M_r$ with $k_{\mathbf S}=0$, for which $\lambda_{8k_{\mathbf S}}(p)=2$, and those with $k_{\mathbf S} \neq 0$.
Hence, combining \eqref{eq:lsd2} and \eqref{eq:lsd1} we have that
\[
\frac{1}{\# \mathcal N_{M,n_0}(x)}  \sum_{n \in \mathcal N_{M,n_0}(x)} \prod_{j=0}^{r-1} \lambda_{4(k_{j+1}-k_j)}(n)= (2\alpha(\vec k))^M \left(g(\vec k;Y) +O\left(\frac{M}{(\log \log x)^2} \mathcal L(\vec k) \right) \right).
\]
\end{proof}

\section{Averaging over $(k_1,\ldots, k_{r-1})$}
The main result 
of this section is the following proposition, which reduces the
computation of the $r$-correlation to a combinatorial expression. For
$\vec k = (k_{1}, \ldots, k_{r-1})\in \mathbb Z^{r-1}$, and using the
convention $k_{0}=k_{r}=0$, let us define 
\begin{equation} \label{eq:elldef}
    \ell_{\vec k}(n_0)=\prod_{j=0}^{r-1} \lambda_{4(k_{j+1}-k_j)}(n_0).
\end{equation}
\begin{prop} \label{prop:corr} Let $A>0$ be fixed. For $1 \le M \le A\log \log x$ we have that
\begin{equation} \label{eq:correxp}
\begin{split}
    \frac{1}{\# \mathcal N_{M,n_0}(x)} \sum_{n \in \mathcal N_{M,n_0}(x)} R_{r,n_0}(n;F_N)
    = \frac{1}{N^{r}} \sum_{\vec k \in \mathbb Z^{r-1}} \widehat f\bigg(\frac{\vec k}{N}\bigg) \ell_{\vec k}(n_0) (2\alpha(\vec k))^M +O\left( \frac{1}{\log \log x}+\frac{1}{M^{10}}\right)
    \end{split}
    \end{equation}
    where $\alpha(\vec k)$ is as given in \eqref{eq:alphadef} and the implied constant depends on
    at most $f,n_0,r$, and $A$.
\end{prop}


To prove the proposition we first apply the main result of the previous section.
In the left-hand side of \eqref{eq:correxp} we apply \eqref{eq:Fform} (see also \eqref{eq:rcorrdef-intro}) and observe that by 
the rapid decay of $\widehat f$ and the trivial bound $|\lambda_{4k}(n)|\le 2^M$ for $n \in \mathcal N_{M,n_0}$ the terms with $\rVert \vec k\lVert_{\infty} \ge N^{1+o(1)}$ contribute $\ll N^{-B}$ for any $B >0$, so that we can add or remove these terms at the cost of a negligible error term. Using the previous observation and applying Proposition \ref{prop:LSD} the left-hand side
of \eqref{eq:correxp}  equals 
\begin{equation}\label{eq:kavg}
\begin{split}
=& \frac{1}{N^r}\sum_{ \vec k \in \mathbb Z^{r-1}} \widehat f \bigg(\frac{\vec k}{N}\bigg) \ell_{\vec k}(n_0) \frac{1}{\# \mathcal N_{M,n_0}(x)} \sum_{n \in \mathcal N_{M,n_0}(x)} \prod_{j=0}^{r-1} \lambda_{4(k_{j+1}-k_j)}(n) \\
=& \frac{1}{N^{r}}\sum_{\vec k \in \mathbb Z^{r-1}}   \widehat f \bigg(\frac{\vec k}{N}\bigg) \ell_{\vec k}(n_0) (2\alpha(\vec k))^M \left(g(\vec k;Y) +O\left(\frac{\mathcal L(\vec k)}{\log \log x} \ \right) \right)+O\left( \frac{1}{N^{10}}\right).
\end{split}
\end{equation}
Since 
$
\alpha(\vec k)=\sum_{ \substack{[\mathbf S] \in \mathcal M_r \\ k_{\mathbf S}=0}} 1
$
  depends on $\vec k$ in a complex way  it is not immediately clear how to perform the average over $\vec k$. We will
decompose $\mathbb R^{r-1}$ into certain ``subspaces'' on which the value
of $\alpha(\vec k)$ is constant 
(outside other such subspaces of lower
dimension.)

\subsection{The decomposition} \label{sec:decomp}
Given
$S\subset \mathcal M_r$ fix an ordering 
$S=\{[\mathbf S_1],\ldots, [ \mathbf S_{|S|}]\}$ as well as choice of
equivalence class representatives $\mathbf S_i$, $i=1,\ldots, |S|$,
and let us define the  
matrix
  $M_S=(a_{ij}) \in \tmop{Mat}_{ |S| \times r-1}(\mathbb R)$ where
\begin{equation} \label{eq:msdef}
a_{ij}=
\begin{cases}
-1 & \text{ if } \quad j \in \mathbf S_i, j-1 \notin \mathbf S_i, \\
0 & \text{ if } \quad j,j-1\in \mathbf S_i \text{ or } j,j-1 \notin \mathbf S_i, \\
1 & \text{ if }\quad  j-1 \in \mathbf S_i, j \notin \mathbf S_i.
\end{cases}
\end{equation}
While $M_S$ depends on the ordering and choice of equivalence class representatives of the elements of $S$ the kernel of $M_S$ does not and since we will only be concerned with the latter this resolves the issue of indeterminacy for us.
For example, for $r=3$ listing the elements of $S=\mathcal M_3$ as $\{[\emptyset], [\{0\}],[\{1\}],[\{2\}]\}$ gives
\begin{equation} \label{eq:m3}
M_{\mathcal M_3}=
\begin{pmatrix}
0 & 0 \\
1 & 0 \\
-1 & 1 \\
0 & -1
\end{pmatrix}.
\end{equation}

The kernels of the matrices $M_S$ as $S$ ranges over subsets of $ \mathcal M_r$ encode the structure of the $r$-correlation of the lattice points. For example, for the triple correlation (i.e. the left-hand side of \eqref{eq:randomeq1} with $r=3$) it turns out that without specifying further conditions on the Schwartz function $f$, such as $f(0,y)=0$, there are five leading order terms.
The kernels of the matrices $M_S$, $S \subset \mathcal M_3$, correspond to the $x,y$-axes, the
line $y=x$, the origin, and all
  of $\mathbb R^2$ and these five subspaces give rise to the five main terms in the triple correlation. After specifying $f$ to detect only \textit{distinct} angles four of the five main terms vanish and the one remaining main term is $\widehat f(0)$, which is consistent with Poisson spacing statistics.
  For larger $r$, directly computing the $r$-correlation in this way becomes difficult  since the number of subspaces that need to be analyzed grows exponentially with $r$. We will pursue a different approach that uses a random model to indirectly solve this combinatorial problem.
\begin{figure}[h] 
\caption{Plot of the kernels of $M_S$ as $S$ varies over $\mathcal M_3$, which correspond to the $x,y$-coordinate axes, the origin, the line $y=x$ and all of $\mathbb R^2$.}
\centering
\includegraphics[width=0.5\textwidth]{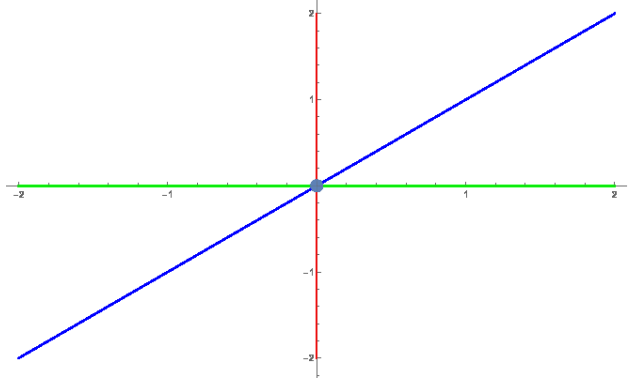}
\label{fig:3corr}
\end{figure}

We now partition $\mathbb R^{r-1}$ into subspaces given by kernels
of $M_S$ as $S$ varies over $S \subset \mathcal M_r$.
Since this is a partition, we need to remove the intersection with other kernels.
We define
\[
\tmop{ker}(M_S)^{\star}=\tmop{ker}(M_S) \setminus \bigcup_{S_1 \subset \mathcal M_r: S_1 \supsetneq S } \tmop{ker}(M_{S_1}).
\]
Let us note that we may have $\tmop{ker}(M_S)=\tmop{ker}(M_{S_1})$ for $S_1 \supsetneq S$, in which case $\tmop{ker}(M_S)^{\star}$ is empty and we say that $S$ is \textit{non-maximal}, we say $S$ is \textit{maximal} if $\tmop{ker}(M_S)$ is non-empty.
We now let  
\begin{equation} \label{eq:vrdef}
\mathcal V_r=\{ \tmop{ker}(M_S)^{\star} : S \subset \mathcal M_r\},
\end{equation}
which gives us our decomposition, into disjoint subsets, 
\begin{equation}\label{eq:decomposition}
\coprod_{V^{\star} \in \mathcal V_r} V^{\star} =\mathbb R^{r-1}.
\end{equation}

\begin{ex} In the case $r=3$, the partition is determined by considering the kernels of the submatrices which consist of collections of rows of
$M_{\mathcal M_3}$ in \eqref{eq:m3} and our partition of $\mathbb R^2$ is
\[
\begin{split}
V_1^{\star}=&\{(x,y) : x \neq 0, y \neq 0, x \neq y\},\\ 
V_2^{\star}=&\{ (x,y) : x=0, y \neq 0\}, \\
V_3^{\star}=&\{(x,y) : y=x, x \neq 0\},\\
V_4^{\star}=&\{ (x,y) : y=0, x \neq 0\}, \\
\end{split}
\]
and $V_5^{\star}=\{(0,0)\}$
 which correspond to the kernels of $M_{\{[ \emptyset] \}}$, $M_{\{[\emptyset],[\{0\}]\}}$, $M_{\{[ \emptyset] ,[\{1\}]\}}$, $M_{\{[ \emptyset] ,[\{2\}]\}}$, $M_{\mathcal M_3}$ respectively, after removing intersections. 
\end{ex}

By construction, for $\alpha(\vec k)$ as defined in \eqref{eq:alphadef}
given $V^{\star} \in \mathcal V_r$ there exists $\alpha_{V^{\star}}
\ge 1$ such that 
\begin{equation} \label{eq:fixed}
\alpha(\vec k) = \alpha_{V^{\star}},  \qquad \forall \vec{k} =
(k_1,\ldots, k_{r-1}) \in V^{\star}. 
\end{equation}
\subsection{Finding cancellation} 

Given $V^{\star} \in \mathcal V_r$ let $ V$ denote the linear span of
the elements 
in $V^{\star}$ and define
\begin{equation}
  \label{eq:d-eq-dim}
d=\tmop{dim}( V).  
\end{equation}
In order to bound $\alpha(\vec k)$ in terms of $d$ we begin with a
simple geometric lemma.
\begin{lem}
\label{lem:hypercube-intersection}
  Let $W \subset \R^{n}$ be a linear subspace of dimension $d$.  Then
  $$
| W \cap \{0,1\}^{n}| \le 2^{d},
$$
i.e., $W$ can intersect the corners of the $n$-dimensional hypercube,
with one corner at the origin, in at most $2^{d}$ points.
\end{lem}
\begin{proof}
  Choose a basis $\vec w_{1},\ldots, \vec w_{d}$ for $W$ and define a
  $n \times d$ matrix $A = [\vec w_{1} \ldots \vec w_{d}]$.  Since $A$ has
  column rank $d$, there exists $d$ independent rows in $A$, and by
  deleting all other rows we may form an invertible $d \times d$
  matrix $B$.  Now, any point $\vec w \in W$ can be written as $\vec w
  = A \vec x$ 
  for a unique $\vec x \in \R^d$, and we further note that
  $A \vec x \in \{0,1\}^{n}$ implies that $B \vec x \in
  \{0,1\}^{d}$. Thus, since 
  $B$ is invertible there are exactly $2^{d}$ possible choices of $\vec x$
  so that $B \vec x \in \{0,1\}^{d}$, and the claimed upper bound follows. \end{proof}



\begin{lem}
  \label{lem:alpha-bound} With $d$ as in \eqref{eq:d-eq-dim} we have
\begin{equation}  \label{eq:alphabd}
\alpha_{V^{\star}} \le 2^{r-d-1}.
\end{equation}
\end{lem}
\begin{proof}
We begin by introducing some convenient notation to parameterize the
sets of $\vec k$'s on which
$\alpha(\vec k)$ is constant.
Define a bilinear form $ \langle \cdot , \cdot \rangle$
on
$\R^{r} \times \R^{r-1}$ by
\begin{equation}\label{eq:bilinear}
\langle \vec l, \vec k \rangle =
\sum_{j=0}^{r-1} l_{j+1}(k_{j+1}-k_{j}).
\end{equation}
Let $W_r=\mathcal P(\{0,1,\ldots, r-1\})$.
For $\mathbf S \in  W_r$ define $\vec l_{\mathbf S}=(l_j) \in \mathbb R^r$ by $l_j=1$ if $j \in \mathbf S$ and $l_j=0$ if $j \notin \mathbf S$.
Observe that
\begin{equation}\label{eq:alpha12}
 2 \cdot \alpha(\vec k) =
\#\{ \mathbf S \in  W_r : \langle \vec l_{\mathbf S}, \vec k\rangle = 0\},
\end{equation}
where the factor of $2$ accounts for the fact that on the right-hand side for each $\mathbf S \in W_r$ we have also counted $\mathbf S^c$ whereas these sets have been identified in $\mathcal M_r$.


Each element $\mathbf S \in W_r$ corresponds to an element of
$\{0,1\}^r \subset \mathbb R^{r}$ under the bijection
$\mathbf S \rightarrow (1_{j \in \mathbf S})_{0 \le j \le r-1}$.
Using this and \eqref{eq:alpha12} we see that
\begin{equation} \label{eq:alpharexpress}
 2  \alpha(\vec k)=\# \{ \vec l \in \{0,1\}^r : \langle \vec l, \vec k \rangle =0 \}.
\end{equation}
Also for $S\subset W_r$ we define $M_S$ as in \eqref{eq:msdef} (the only difference is that we do not need to fix equivalence class representatives) and also define $\tmop{ker}(M_S)^{\star}$ analogously (which is the same as before since choosing equivalence class representatives does not affect the kernels.)
Additionally, under the bijection described above each $S \subset W_r$
is associated to some $\widetilde S \subset \{0,1\}^r$ so that  we can
re-express $\tmop{ker}(M_S) \subset \R^{r-1}$ as 
\begin{equation} \label{eq:kerexpress}
\tmop{ker}(M_S)=
\{ \vec k \in \R^{r-1} : \langle \vec l, \vec k \rangle  = 0, \, \,
\forall \vec l \in \widetilde S\}. 
\end{equation}
Let
$$
\tmop{ker}(M_S)^{\perp} :=
\{ \vec l \in \R^r : \langle \vec l, \vec k \rangle = 0, \, \, \forall
\vec k \in \tmop{ker}(M_S) \}. 
$$

Note that if $d = \dim(\tmop{ker}(M_S))$, then $\dim(\tmop{ker}(M_S)^{\perp}) = r-d$.  To
see this, write $\langle \vec l, \vec k \rangle = \vec k \cdot
\phi(\vec l)$, where 
$\cdot$ denotes the standard inner product on $\R^{r-1}$, and
$\phi : \R^r \to \R^{r-1}$ is given by the {\em surjection}
$\phi( \vec l) = (l_{2}-l_{1}, l_{3}-l_{2}, \ldots, l_{r}-l_{r-1})$.
Further, if $S$ is maximal and $\vec k \in \tmop{ker}(M_S)^{\star}$, then using \eqref{eq:alpharexpress} and \eqref{eq:kerexpress} we have 
$2 \alpha(\vec k) = \# \widetilde S$; by Lemma~\ref{lem:hypercube-intersection}
$\# \widetilde S \leq 2^{r-d}$ since $\dim(\tmop{ker}(M_{S})^{\perp}) = r-d$, and the proof is
concluded.
\end{proof}

By Lemma \ref{lem:alpha-bound} we have $2\alpha_{V^{\star}} \le
2^{r-d}$ and hence, by \eqref{eq:kavg}, \eqref{eq:decomposition}, and
\eqref{eq:fixed}, to establish Proposition
\ref{prop:corr} it suffices to show for each $V^{\star} \in \mathcal
V_r$ that
\begin{equation} \label{eq:poisson}
\begin{split}
    &\sum_{\vec k \in V^{\star} \cap \mathbb Z^{r-1}}   \widehat f \bigg( \frac{\vec k}{N} \bigg)  \ell_{\vec k}(n_0) \bigg(g(\vec k;Y) +O\bigg(\frac{\mathcal L(\vec k)}{\log \log x} \ \bigg) \bigg)\\
&\qquad \qquad \qquad \qquad \qquad = \sum_{\vec k \in  V^{\star} \cap \mathbb Z^{r-1} }   \widehat f \bigg( \frac{\vec k}{N} \bigg) \ell_{\vec k}(n_0)  +O\left(\frac{N^{d}}{\log \log x}+\frac{N^d}{M^{10}} \ \right).
\end{split}
\end{equation}
Using Proposition \ref{prop:approx} we can express $g(\vec k;Y)$ and $\mathcal L(\vec k)$ in terms of short Dirichlet polynomials of lengths $y=N^{o(1)}$ and since $V \cap \mathbb Z^{r-1}$ is a lattice this will allow us to use Poisson summation to establish \eqref{eq:poisson}.


We first require that the angles $(\theta_p)_{p \equiv 1 \Mod 4}$ are linearly independent $\Mod {2\pi}$ over $\mathbb Q$. We also need a quantitative bound for how close these combinations are to multiples of $2\pi$, which is a consequence of repulsion of angles of Gaussian integers.

\begin{lem} \label{lem:repulsion} Let $J \ge 1$. For each
  $j=1, \ldots, J$ suppose $p_j \equiv 1 \Mod 4$, and
  $p_{i} \neq p_{j}$ if $i \neq j$.
  Then for $c_1,\ldots, c_j \in \mathbb Z$
\begin{equation}\label{eq:indp}
\exp\left(i \sum_{j=1}^J c_j \theta_{p_j} \right) \neq 1
\end{equation}
unless $c_j=0$ for each $j=1, \ldots, J$. Moreover,
\begin{equation} \label{eq:repulsion}
\bigg|\exp\left(i \sum_{j=1}^J c_j \theta_{p_j} \right)-1 \bigg|\ge \frac{1}{ \sqrt{ p_1^{|c_1|} \cdots p_J^{|c_J|}}}.
\end{equation}
\end{lem}

\begin{proof} WLOG assume $c_j> 0$ (replacing $\theta_p$ with
  $-\theta_p$ if necessary) for each $j$ and recall that
  $p_1, \ldots, p_J$ are distinct. We split the proof into two
  cases. First if $J=1$ we cannot have that $e^{i c \theta_p}=1$ since
  $\theta_p$ is not a rational multiple of $\pi$ by Niven's theorem
  \cite[Th'm 3.11]{niven}.

Let $m=p_2^{c_2/2} \cdots p_J^{c_J/2}$.  For $J \ge 2$ we observe that
if equality in \eqref{eq:indp} holds then there exist
$x+iy,u+iv \in \mathbb Z[i]$ such that
\[
\frac{x+iy}{p_1^{c_1/2}}=\exp\left(i c_1 \theta_p \right)=\exp\left(-i\sum_{j=2}^J c_j \theta_{p_j} \right) =\frac{u+iv}{m}
\]
which implies that
\[
x+iy = \frac{p_1^{c_1/2}}{m} (u+iv)
\]
consequently $c_j$ is even for each $j$ so $m \in \mathbb Z$ and since $(m,p_1)=1$, $m$ divides both $u$ and $v$. Also $m^2=u^2+v^2$ so that $u$ or $v$ equals zero. This implies that $x$ or $y$ equals zero but this is not possible since it would imply that $e^{ic \theta_{p_1}}=1$ for some integer $c$. This proves the first claim.

 Write $\theta=\sum_{j=1}^J c_j \theta_{p_j}$. Assume $e^{i\theta} \neq \pm 1$. There exists $a+ib \in \mathbb Z[i]$ with $a^2+b^2=p_1^{c_1}\cdots p_J^{c_J}$ and $\frac{a+ib}{|a+ib|}=e^{i\theta}$.  Hence,
\[
0<|1-e^{i\theta}|=\bigg|\frac{e^{i\theta}-e^{-i\theta}}{1+e^{i\theta}}\bigg| = \frac{1}{|1+e^{i\theta}|} \cdot \frac{2|b|}{\sqrt{p_1^{c_1}\cdots p_J^{c_J}}}
\]
so that $|b| \ge 1$ and consequently
\[
|1-e^{i\theta}| \ge \frac{1}{ \sqrt{p_1^{c_1}\cdots p_J^{c_J}}}.
\]
\end{proof}

Before proceeding to the next lemma let us recall the bilinear form defined in \eqref{eq:bilinear}. Given $\vec l=(l_1,\ldots, l_{r}) \in \mathbb Z^{r}$   let
\[
k_{\vec l}:=\langle \vec l, \vec k \rangle =\sum_{j=0}^{r-1} l_{j+1} (k_{j+1}-k_j) \in \Z.
\]
Recall that for $\vec l$ such that $l_{j+1}=1$ if $j \in \mathbf S$ and $l_{j+1}=0$ if $j \notin \mathbf S$ for some $[\mathbf S] \in \mathcal M_r$ we have $k_{\vec l}=k_{\mathbf S}$.

\begin{lem} \label{lem:cancellation}
Let $H$ be a Schwartz function and $w > v \ge 0$ be integers.
Let $p_1,\ldots, p_w$ be distinct primes $p_t \equiv 1 \tmod 4$, $t=1,\ldots, w$.
Additionally, let $a_1,\ldots, a_w$ be nonzero integers and $\vec
l_1,\ldots, \vec l_w \in \mathbb Z^{r-1}$ be such that $\lVert \vec l_t
\rVert_{\infty} \le |a_t|$ for each $t=1,\ldots,w$. Suppose
$p_1^{|a_1|}\cdots p_w^{|a_w|} \le N^{\delta}$ for some sufficiently small
$\delta>0$. 
Then we have that
\[
\sum_{\substack{\vec k \in V^{\star} \cap \mathbb Z^{r-1} \\ k_{\vec l_t} \neq 0, t=v+1, \ldots, w}}  \exp\bigg(4i\sum_{t=1}^w k_{\vec l_t} \theta_{p_t}\bigg) H\bigg( \frac{\vec k}{N}\bigg)\ll N^{d-1+o(1)}
\]
where the implied constant depends at most on $r$ and $H$.
\end{lem}
\begin{proof}
  Recall for each $V^{\star} \in \mathcal V_r$ that
  $ V=\tmop{ker}(M_S)$ for some $S \subset \mathcal M_r$.
%
  Since a submodule $M$ of a finitely generated $\Z$-module of rank
  $r-1$  is also free, and of rank $d \leq r-1$ (cf. \cite[Proposition
  9.7]{rotman-advanced-modern-algebra-book}), there exists a basis
  $B=\{\mathbf b_1,\ldots,\mathbf b_d\}$ for $ V \cap \mathbb Z^{r-1}$ such
  that
\begin{equation} \label{eq:Bdef}
   V \cap \mathbb Z^{r-1}=\bigg\{ \sum_{\ell=1}^d \mathbf b_{\ell} v_{\ell} : (v_1,\ldots, v_d) \in \mathbb Z^d \bigg\}, \qquad \mathbf b_{\ell}=\begin{pmatrix} 
b_{\ell,1} \\
\vdots \\
b_{\ell,(r-1)} 
\end{pmatrix},
\end{equation}
and $b_{\ell,j} \in \mathbb Z$ for each $\ell=1,\ldots,d$,
$j=1,\ldots, r-1$. Define
$c_{\ell,\vec l}=\sum_{j=0}^{ r-1} 
l_{j+1}(b_{\ell, (j+1)}-b_{\ell,j})$ where $l_{1}, \ldots, l_{r-1}$
denotes the components of $\vec l$,
and where we use the convention
$b_{\ell,0}=b_{\ell,r}=0$; note that $b_{\ell,j} =O_{r}(1)$,
and thus $c_{\ell,\vec l} =O_{r}(\lVert \vec l \rVert_{\infty})$ since the number
of vector spaces $V$ is $O_{r}(1)$
(cf. Lemma~\ref{lem:bounded-number-of-W}.) 
Using \eqref{eq:Bdef} we have, 
with $k_{j}$ denoting the $j$-th coordinate of $\vec{k} \in V \cap
\Z^{r-1}$, that 
$k_j=\sum_{\ell=1}^d b_{\ell,j}v_{\ell}$ for each $j=1,\ldots, r-1$,
and $k_{\vec l}=\sum_{\ell=1}^d v_{\ell} c_{\ell,\vec l}$ for each
$\vec l \in \mathbb Z^{r}$. 
Let
\[
G(x_1, \dots, x_{d})= H\left( \sum_{\ell=1}^d b_{\ell,1}x_{\ell}  , \ldots, \sum_{\ell=1}^d b_{\ell,{r-1}}x_{\ell}  \right).
\]
Observe that $G: \mathbb R^d \rightarrow \mathbb R$ is a Schwartz
function. Also, $\sum_{\vec k \in ( V \setminus V^{\star}) \cap
  \mathbb Z^{r-1}} |H(\vec k/N)| \ll N^{d-1+o(1)}$ as $
V\setminus V^{\star}$ consists of lower dimensional subspaces.
Hence, we have that
\begin{equation} \label{eq:expand}
\begin{split}
& \sum_{\substack{\vec k \in V^{\star} \cap \mathbb Z^{r-1} \\ k_{\vec l_t} \neq 0 \\ t=v+1, \ldots, w}}  \exp\bigg(4i\sum_{t=1}^w k_{\vec l_t} \theta_{p_t}\bigg) H\bigg( \frac{\vec k}{N}\bigg) 
= \sum_{\substack{\vec k \in V^{} \cap \mathbb Z^{r-1} \\ k_{\vec l_t} \neq 0 \\ t=v+1, \ldots, w}}  \exp\bigg(4i\sum_{t=1}^w k_{\vec l_t} \theta_{p_t}\bigg) H\bigg( \frac{\vec k}{N}\bigg) +O(N^{d-1+o(1)}) \\
&=\sum_{\substack{\vec v \in \mathbb Z^{d} \\ \sum_{\ell=1}^d v_{\ell}c_{\ell,\vec l_t} \neq 0 \\ t=v+1, \ldots, w}} \prod_{\ell=1}^d \bigg( \exp\bigg(4i v_{\ell} \sum_{t=1}^w c_{\ell,\vec l_t} \theta_{p_t}\bigg) \bigg) G\bigg( \frac{\vec v}{N}\bigg)+O(N^{d-1+o(1)}).
\end{split}
\end{equation}
Since $\sum_{\ell=1}^d v_{\ell} c_{\ell,\vec l_t} \neq 0$, WLOG assume $c_{1,\vec l_w} \neq 0$.
Write $E$ for the subset of $\mathbb Z^d$ such that $\sum_{\ell=1}^d v_{\ell} c_{\ell,\vec l_t} \neq 0$ for each $t=v+1,\ldots, w$. The left-hand side  of \eqref{eq:expand} is
\begin{equation}\label{eq:poiss1}
\ll  \sum_{|v_2|, \ldots, |v_d| \ll N^{1+o(1)}} \bigg| \sum_{v_1 \in \mathbb Z} 1_{\vec v \in E}  G \bigg(\frac{ v_1}{N}, \frac{v_2}{N}, \ldots, \frac{v_d}{N} \bigg) \exp\left( 4i v_{1} \sum_{t=1}^w c_{\ell,\vec l_t} \theta_{p_t}\right)    \bigg|+N^{d-1+o(1)}.
\end{equation}
 We wish to extend the inner sum to all of $\mathbb Z$. To do this,
 for each $t =1,\ldots, w$ with $c_{1,\vec l_t} \neq 0$ we need to add back in the
 point $v_{1}$ such that $c_{1,\vec l_t} v_1= -\sum_{\ell=2}^d
 v_{\ell} c_{\ell,\vec l_t}$. Since there are $\le w=N^{o(1)}$ such
 points,
 we can extend the inner sum to all of $\mathbb Z$ at the cost of an
 error term of size 
\begin{equation} \label{eq:poiss2}
\ll N^{o(1)}
 \sum_{|v_2|, \ldots, |v_d| \ll N^{1+o(1)}} 1 \ll N^{d-1+o(1)}.
\end{equation}
Let $\theta=4  \sum_{t=1}^w c_{\ell,\vec l_t} \theta_{p_t}$ and WLOG we may assume $-\pi \le \theta < \pi$ (since $v_1 \in \mathbb Z$.) 
Applying Poisson summation 
  and using \eqref{eq:indp} (which implies $e^{i\theta} \neq 1$) we have that
\[
\begin{split}
\sum_{v_1 \in \mathbb Z}  G \left(\frac{v_1}{N},\frac{v_2}{N}, \ldots, \frac{v_d}{N} \right) e^{i\theta v_1} =&N \sum_{a \in \mathbb Z} \int_{\mathbb R}  G \left(x, \frac{v_2}{N}, \ldots, \frac{v_d}{N} \right) e^{iN(\theta-2\pi a) x } \, d x \\
 \ll &  N \sum_{a \in \mathbb Z} \frac{1}{N^B |\theta-2\pi a|^B}
 \end{split}
\]
for any integer $B \ge 0$. Since  $c_{\ell,\vec l_{t}} \ll |a_t|$, we have by \eqref{eq:repulsion} and the assumption $p_1^{|a_1|}\cdots p_w^{|a_w|} \le N^{\delta}$ that $|\theta|  \gg N^{-1/2}$ since $\delta$ is sufficiently small. Hence, since $-\pi \le \theta < \pi$ this implies that the right-hand side above is 
\[
\ll N^{1-B} \frac{1}{|\theta|^B} \ll N^{-100}.
\]
Combining this estimate with \eqref{eq:poiss1} and \eqref{eq:poiss2} completes the proof.
\end{proof}
\begin{lem} \label{lem:errorbd}
For each $V^{\star} \in \mathcal V_r$ with $d=\tmop{dim}(V)$ we have that
 \[
 \sum_{\vec k \in V^{\star} \cap \mathbb Z^{r-1}} \mathcal L(\vec k) \,\, \bigg| \widehat f\bigg( \frac{\vec k}{N}\bigg)\bigg| \mathcal   \ll N^{d},
\]
with $\mathcal L(\vec k)$ as given in
\eqref{eq:glkdef}. The implied constant depends on at most $A', n_0$, and $f$.
\end{lem}
\begin{proof}
Observe that there exists a Schwartz function $H$ with $|\widehat f| \le H$.
Recalling
\eqref{eq:glkdef} it suffices to show that 
\begin{equation} \label{eq:reduction}
\sum_{\substack{\vec k \in V^{\star} \cap \mathbb Z^{r-1} \\ k_{\mathbf S} \neq 0}}
L(1,8k_{\mathbf S})^u  H\bigg( \frac{\vec k}{N} \bigg)  \ll N^d 
\end{equation}
for any $[\mathbf S] \in \mathcal M_r$ and $u \in \mathbb Z$ fixed, $u \neq
0$ (note that $L(1,4k) >0$ for any $k \in \mathbb Z$, $k \neq 0$.) Since $H$ is a Schwartz function the sum is effectively restricted
to $\lVert \vec k\lVert_{\infty} \le N^{1+o(1)}$. Applying part~2 of Lemma
\ref{prop:approx} with
$y=e^{(\log N)^{3/4}}$ we have for $|k_{\mathbf S}| \le 
N^{1+o(1)}$ that 
\begin{equation} \label{eq:lexpand}
L(1, 8k_{\mathbf S})^u=\sum_{m \le y} \frac{\lambda_{8k_{\mathbf S}}(m;u)}{m}+O\left(e^{-(\log N)^{1/12-o(1)}} \right).
\end{equation}
Applying \eqref{eq:lexpand} we see that the left-hand side of \eqref{eq:reduction} equals
\begin{equation}\label{eq:expandl}
\sum_{m \le y} \sum_{\substack{\vec k \in V^{\star} \cap \mathbb Z^{r-1} \\ k_{\mathbf S} \neq 0}} \frac{\lambda_{8k_{\mathbf S}}(m;u)}{m} H\bigg(\frac{\vec k}{N}\bigg)+O\bigg(N^d e^{-(\log N)^{-1/12+o(1)}}\bigg),
\end{equation}
where we have used that 
\begin{equation}\label{eq:poissonlattice}
\sum_{\vec k \in  V \cap \mathbb Z^{r-1}} H\bigg(\frac{\vec k}{N}\bigg) \ll N^d 
\end{equation}
to estimate the error term. To see why \eqref{eq:poissonlattice}
holds, note that $ V \cap \mathbb Z^{r-1}$ is a lattice (of rank $d$),
so by applying Poisson summation the bound follows.  Using the
definition of $\lambda_{8k}(m;u)$ as given in \eqref{eq:lzdef}
we have for each $m \le y$ that
\[
\sum_{\substack{\vec k \in V^{\star} \cap \mathbb Z^{r-1} \\ k_{\mathbf S} \neq 0}} \lambda_{8k_{\mathbf S}}(m;u)H\bigg( \frac{\vec k}{N}\bigg)=\sum_{\mathfrak b \subset \OO: |\OO/\mathfrak b|=m} b_u(\mathfrak b) \sum_{\substack{\vec k \in V^{\star} \cap \mathbb Z^{r-1} \\ k_{\mathbf S} \neq 0}} \Xi_{8k_{\mathbf S}}(\mathfrak b) H\bigg( \frac{\vec k}{N} \bigg).
\]

We first consider $m$ for which there is cancellation in the
$\vec{k}$-sum, namely 
 $m \le y$ which are not of the form $l^2$ or $2l^2$. 
 For each $\mathfrak b \subset \OO$ with $|\OO/\mathfrak b|=m$ by considering the factorization of $\mathfrak b$ into prime ideals, we see that in this case there exist distinct primes $p_1,\ldots, p_w$ each $\equiv1 \tmod 4$ and non-zero integers $a_1,\ldots, a_w$ such that $\Xi_{8k}(\mathfrak b)=e^{i8k \sum_{j=1}^w a_j \theta_{p_j}}$ and $p_1^{|a_1|} \cdots p_w^{|a_w|} \le y$ (note that $\Xi_{8k}((1+i))=1$.)
Hence using these observations and recalling $b_u(\mathfrak b) \ll |\OO/\mathfrak b|^{o(1)}$, applying Lemma \ref{lem:cancellation} with $v=0$ and $\vec l_j$ such that $k_{\vec l_j}=2a_j k_{\mathbf S}$ for each $j=1,\ldots, w$, gives for $m\le y$, which is not of the form $m=l^2,2l^2$, that
\begin{equation}\label{eq:one}
\bigg| \sum_{\substack{\vec k \in V^{\star} \cap \mathbb Z^{r-1} \\ k_{\mathbf S} \neq 0}} \lambda_{8k_{\mathbf S}}(m;u) H\bigg(\frac{\vec k}{N}\bigg)  \bigg| \ll N^{d-1+o(1)}.
\end{equation}

To bound the contribution of the terms with $m=l^2,2l^2$ we use that $\lambda_{8k_{\mathbf S}}(m;u)\ll m^{o(1)}$ along with \eqref{eq:poissonlattice}. Combining this with \eqref{eq:one} we conclude that the main term in the left-hand side of \eqref{eq:expandl} is
\[
\ll N^d \sum_{l^2 \le y}\frac{1}{l^{2-o(1)}} + N^{d-1+o(1)} \ll N^d,
\]
which completes the proof.
\end{proof}

\begin{lem} \label{lem:errorbd2} Let $V^{\star} \in \mathcal V_R$ and $m$ be a square free integer such that $(m,2n_0)=1$ and
  $2 \le m \le N^{\delta}$, where $\delta$ is sufficiently
  small. We have
  that
 \begin{equation}\label{eq:maintermcancellation}
 \sum_{\vec k \in V^{\star} \cap \mathbb Z^{r-1}} \ell_{\vec k}(n_0)s(m;\vec k) \,\, \widehat f\bigg( \frac{\vec k}{N}\bigg) \mathcal   \ll N^{d-1+o(1)},
 \end{equation}
 where $\ell_{\vec k}(n_0)$ is as defined in \eqref{eq:elldef} and
 $s(m;\vec k)$ is  as in \eqref{eq:wsdef}. The implied constant depends
 at most on $r,n_0$ and $f$. 
 \end{lem}

\begin{proof}
 Let us write $n_0=2^a
 q_1^{2b_1}\cdots q_u^{2b_u} p_1^{a_1} \cdots p_v^{a_v}$ where
 $a,b_1,\ldots, b_u,a_1,\ldots, a_v$ are non-negative integers, $q_j$
 are primes $\equiv 3 \Mod 4$, $j=1, \ldots, u$, and $p_j$ are distinct
 primes  $\equiv 1 \Mod 4$, $j=1, \ldots, v$. 
Plainly, $\lambda_{4k}(q^{2b})=1$ for any $b \ge 0$ and prime $q\equiv 3 \Mod 4$.
Also since $(2)=((1+i)^2)$ as ideals in $\OO$ we have $\lambda_{4k}(2^a)=(-1)^{ka}$ so
using that $\sum_{j=0}^{r-1} (k_{j+1}-k_j)=0$ we have $\prod_{j=0}^{r-1} \lambda_{4(k_{j+1}-k_j)}(2^a)=1$.
Additionally, for $p\equiv 1 \tmod 4$ let $\pi \in \mathbb Z[i]$ be such that $\pi=\sqrt{p} e^{i\theta_p}$.
Observe that
\[
  \lambda_{4k}(p^a)=
  \frac{1}{p^{2ka}} \sum_{l=0}^a \pi^{4kl} \overline  \pi^{4k(a-l)}=
  e^{-4kai\theta_p} \sum_{l=0}^a e^{8kl i\theta_p}.
\]
Combining the observations above and recalling that $\sum_{j=0}^{r-1}
(k_{j+1}-k_j)=0$ we have that
\begin{equation} \label{eq:ellform}
\begin{split}
\ell_{\vec k}(n_0)&= \prod_{t=1}^v \prod_{j=0}^{r-1}
\lambda_{4(k_{j+1}-k_j)}(p_t^{a_t})= \prod_{t=1}^v  \prod_{j=0}^{r-1}  \sum_{0 \le l \le a_t} e^{8l (k_{j+1}-k_j) i\theta_{p_t}} \\
&= \prod_{t=1}^v \sum_{0 \le l_1, \ldots, l_r \le a_t} \exp\bigg(8i \sum_{j=0}^{r-1} l_{j+1}(k_{j+1}-k_j) \theta_{p_t} \bigg) =\sum_{\substack{\vec l_1, \ldots,
    \vec l_v \in \mathbb Z_{\ge 0}^{r} \\ \lVert \vec l_t
    \rVert_{\infty}\le a_t, t=1,\ldots,v }} \exp\bigg(8i \sum_{t=1}^v
k_{\vec l_t} \theta_{p_t} \bigg). 
\end{split}
\end{equation}

We can write $m=p_{v+1} \cdots p_w$ where $p_{v+1},\ldots,p_w$ are distinct primes $\equiv 1 \tmod 4$, which are co-prime to $n_0$. 
Let $W=\{1,\ldots, w-v\}$ and for $S \subset W$ let $\epsilon_{j,S}=1$ if $j \in S$ and $\epsilon_{j,S}=-1$ if $j \notin S$.   Recalling that $\lambda_{8k}(p)=e^{8ki\theta_p}+e^{-8ki\theta_p}$ for $p \equiv 1 \pmod 4$ we have that
\begin{equation} \label{eq:sform}
\begin{split}
s(m;\vec k)=& \sum_{\substack{[\mathbf S_{v+1}],\ldots, [\mathbf S_{w}] \in \mathcal M_r \\ k_{\mathbf S_j} \neq 0, \,  j=v+1,\ldots, w}} \lambda_{8k_{\mathbf S_{v+1}}}(p_{v+1}) \cdots \lambda_{8k_{\mathbf S_w}}(p_w) \\
=& \sum_{\substack{[\mathbf S_{v+1}],\ldots, [\mathbf S_{w}] \in \mathcal M_r \\ k_{\mathbf S_l} \neq 0, \,  l=v+1,\ldots, w}} \sum_{S \subset W} \exp\left( 8i\sum_{j=v+1}^w \epsilon_{j,S} k_{\mathbf S_j} \theta_{p_j} \right) .
\end{split}
\end{equation}

To complete the proof we combine \eqref{eq:ellform} and
\eqref{eq:sform} then use the resulting expression in the left-hand
side of \eqref{eq:maintermcancellation} and apply Lemma
\ref{lem:cancellation}, with $k_{\vec l_{t}}=2\epsilon_{j,S}k_{\mathbf S_{t}}$ for each
$t=v+1,\ldots, w$. The hypotheses of the lemma are satisfied since
$n_0m \le N^{2\delta}$. Finally note that $2^{|W|}=
N^{o(1)}$ since $m \le N^{\delta}$.
\end{proof}

\subsection{Proof of Proposition \ref{prop:corr}}

\begin{proof}[Proof of Proposition \ref{prop:corr}]
Let $V^{\star} \in \mathcal V_r$ with $d =\tmop{dim}(V)$. Applying Lemma \ref{lem:errorbd} we have
that
\begin{equation}\label{eq:errorbdlk}
\begin{split}
&\sum_{\vec k \in V^{\star} \cap \mathbb Z^{r-1}}  \ell_{\vec k}(n_0) \widehat f \bigg( \frac{\vec k}{N} \bigg)  \left(g(\vec k;Y) +O\left(\frac{\mathcal L(\vec k)}{\log \log x} \ \right) \right)\\
&\qquad \qquad \qquad \qquad =\sum_{\vec k \in V^{\star} \cap \mathbb Z^{r-1}}   \ell_{\vec k}(n_0) g(\vec k;Y)  \widehat f \bigg( \frac{\vec k}{N} \bigg)  +O\left(\frac{N^d}{\log \log x}  \right).
\end{split}
\end{equation}
Since $\widehat f$ is a Schwartz function the sum is effectively restricted to $\lVert \vec k \rVert_{\infty} \le N^{1+o(1)}$. Recall the definitions of $w(m;Y),s(m;\vec k)$ as given in \eqref{eq:wsdef}.
By Proposition \ref{prop:approx} with $y=e^{(\log N)^{3/4}}$ and \eqref{eq:fixed} we have that
\begin{equation} \label{eq:gexpand}
\begin{split}
\sum_{\vec k \in V^{\star} \cap \mathbb Z^{r-1}} \ell_{\vec k}(n_0)  g(\vec k;Y)  \widehat f \bigg( \frac{\vec k}{N} \bigg) =&\sum_{m \le y} \frac{\alpha_{V^{\star}}^{-\Omega(m)}w(m;Y)}{m}  \sum_{\vec k \in V^{\star} \cap \mathbb Z^{r-1}} \ell_{\vec k}(n_0)s(m;\vec k) \widehat f\bigg(\frac{\vec k}{N} \bigg)\\
&+O\bigg(N^{d} e^{-(\log N)^{1/12-o(1)}}\bigg),
\end{split}
\end{equation}
where we have used \eqref{eq:poissonlattice} to bound the error term
(recall $H$ is Schwartz function with $|\widehat f| \le H$.) Using Lemma \ref{lem:errorbd2}, the contribution of the terms with $2 \le m \le y$ is $\ll y N^{d-1+o(1)} \ll N^{d-1+o(1)}$.
Hence, applying this observation in \eqref{eq:gexpand} we see that
the sum on the right-hand side of \eqref{eq:errorbdlk} equals
\[
\sum_{\vec k \in V^{\star} \cap \mathbb Z^{r-1}} \widehat f \bigg( \frac{\vec
  k}{N}\bigg) \ell_{\vec k}(n_0) +O\bigg(N^d e^{-(\log N)^{1/12-o(1)}} \bigg). 
\]
This along with \eqref{eq:errorbdlk} establishes \eqref{eq:poisson},
which, on taking
\eqref{eq:kavg} into account, completes the proof of 
Proposition \ref{prop:corr}.
\end{proof}

\subsection{The variance}
Let $h: \mathbb R^{2r-1} \rightarrow \mathbb R$ be a Schwartz function with 
\begin{equation} \label{eq:hdef}
h(x_1,\ldots, x_{r-1},0,x_{r+1}, \ldots, x_{2r-1})=\widehat f(x_1,\ldots,x_{r-1})\widehat f(x_{r+1},\ldots, x_{2r-1}).
\end{equation}
Additionally, for $\vec k \in \mathbb Z^{2r-1}$ with $\vec k=(\vec k_1,0,\vec k_2)$ and $\vec k_{1},\vec k_2 \in \mathbb Z^{r-1}$ let
\begin{equation} \label{eq:ltildedef}
\widetilde \ell_{\vec k}(n_0)=\ell_{\vec k_1}(n_0)\ell_{\vec k_2}(n_0).
\end{equation}

\begin{prop}  \label{prop:var}
Let $A>0$ be fixed.
Suppose that $1 \le M \le A \log \log x$.
We have that
\begin{equation} \label{eq:var}
\frac{1}{\# \mathcal N_{M,n_0}(x)} \sum_{n \in \mathcal N_{M,n_0}(x)} R_{r,n_0}(F_N;n)^2=\frac{1}{N^{2r}} \sum_{\substack{\vec k \in \mathbb Z^{2r-1} \\ k_r=0 }} (2\alpha(\vec k))^M \widetilde \ell_{\vec k}(n_0)h\bigg( \frac{\vec k}{N}\bigg)+O\left(\frac{1}{\log \log x}+\frac{1}{M^{10}} \right),
\end{equation}
where the implied constant depends at most on $f,r,n_0$ and $A$.
\end{prop}
 
\begin{proof}
  Using \eqref{eq:Fform} (see also (\ref{eq:elldef}) and the first
  line of (\ref{eq:kavg})) the left-hand side of \eqref{eq:var} equals
\[
\begin{split}
=&\frac{1}{N^{2r}}\sum_{\vec k, \vec l \in \mathbb Z^{r-1}} \widehat f\bigg(\frac{\vec k}{N} \bigg) \widehat f\bigg( \frac{\vec l}{N}\bigg) \ell_{\vec k}(n_0)\ell_{\vec \ell}(n_0) \frac{1}{\# \mathcal N_{M,n_0}(x)} \sum_{n \in \mathcal N_{M,n_0}(x)} \prod_{j=0}^{r-1} \lambda_{4(k_{j+1}-k_j)}(n)\lambda_{4(l_{j+1}-l_j)}(n) \\
=&\frac{1}{N^{2r}}\sum_{\substack{\vec k \in \mathbb Z^{2r-1} \\ k_r=0}} h\bigg( \frac{\vec k}{N}\bigg) \widetilde \ell_{\vec k}(n_0) \frac{1}{\# \mathcal N_{M,n_0}(x)} \sum_{n \in \mathcal N_{M,n_0}(x)} \prod_{j=0}^{2r-1} \lambda_{4(k_{j+1}-k_j)}(n),
\end{split}
\]
where we have also used the convention $k_0=k_{2r}=0$ in the last line.
As before, recall that the contribution of the terms with $\lVert \vec k \rVert_{\infty} \ge N^{1+o(1)}$ is $\ll N^{-B}$ for any $B>0$, which is negligible so we can add or remove these terms as we wish. 

We now use an argument similar to that given in the proof of Proposition \ref{prop:corr}. As before we apply Proposition \ref{prop:LSD} to the inner sum on the right-hand side above (using the previous observation to restrict to $\lVert \vec k \rVert_{\infty} \le N^{1+o(1)}$.) The next step differs slightly; we apply our decomposition \eqref{eq:decomposition} to $\mathbb R^{2r-1}$ (as opposed to $\mathbb R^{r-1}$ previously) and we have the additional constraint $k_r=0$.
This gives that the right-hand side of the above equation equals
\begin{equation} \label{eq:varexpand}
\begin{split}
&\frac{1}{N^{2r}} \sum_{\substack{\vec k \in \mathbb Z^{2r-1} \\ k_r=0}} h\bigg( \frac{\vec k}{N}\bigg) \widetilde \ell_{\vec k}(n_0) (2\alpha(\vec k))^M\left(g(\vec k;Y)+O\left(\frac{\mathcal L(\vec k)}{\log \log x} \right) \right)+O(N^{-10})\\
&= \frac{1}{N^{2r}} \sum_{V^{\star} \in \mathcal V_{2r}}  \sum_{\substack{\vec k \in V^{\star} \cap \mathbb Z^{2r-1} \\ k_r=0}} h\bigg( \frac{\vec k}{N}\bigg) \widetilde \ell_{\vec k}(n_0) (2\alpha(\vec k))^M  \left(g(\vec k;Y)+O\left(\frac{\mathcal L(\vec k)}{\log \log x} \right) \right)+O(N^{-10}),\\
\end{split}
\end{equation}
where in the first line we added back in the terms with $\lVert \vec k \lVert_{\infty} \ge N^{1+o(1)}$ after applying Proposition \ref{prop:LSD}. Similar to before, our strategy is to evaluate the inner sum for each $V^{\star} \in \mathcal V_{2r}$, however we need to account for the condition $k_r=0$ (and we are also working with higher dimensional lattices.) We next note that the condition $k_{r}=0$ can be imposed by
adding an extra linear relation to $S$.
Namely, let $\mu=[\{0,\ldots,r-1\} ]\in \mathcal M_{2r}$ and observe
that $k_{\mu}=\sum_{j=0}^{r-1} (k_{j+1}-k_j)=k_r$. For $S \subset
\mathcal M_{2r}$ let $\widetilde S=S\cup \{ \mu
\}$. Recall that for 
each $V^{\star} \in \mathcal V_{2r}$ there exists $S \subset \mathcal
M_{2r}$ 
such that $V^{\star}=\tmop{ker}(M_S)^{\star}$.  Set $\widetilde V^{\star}=\tmop{ker}(M_{\widetilde S})^{\star}$, $d=\tmop{dim}(\widetilde V)$, and note that $d \le 2r-1$. Hence the right-hand side of \eqref{eq:varexpand} equals
\[
\frac{1}{N^{2r}} \sum_{V^{\star} \in \mathcal V_{2r}}  (2\alpha_{\widetilde V^{\star}})^M \sum_{\substack{\vec k \in \widetilde V^{\star} \cap \mathbb Z^{2r-1} }} \widetilde \ell_{\vec k}(n_0) h\bigg( \frac{\vec k}{N}\bigg) \left(g(\vec k;Y)+O\left(\frac{\mathcal L(\vec k)}{\log \log x} \right) \right)+O(N^{-10}).
\]
Since $\widetilde V^{\star} \in \mathcal V_{2r}$ arguing as in the proof of Proposition \ref{prop:corr} gives that
\[
\begin{split}
&\sum_{\vec k \in \widetilde V^{\star} \cap \mathbb Z^{2r-1}} \widetilde \ell_{\vec k}(n_0) h\bigg(\frac{\vec k}{N} \bigg)  \bigg( g(\vec k; Y)+O\left( \frac{\mathcal L(\vec k)}{\log \log x} \right)\bigg)\\
&\qquad \qquad \qquad =\sum_{\vec k \in \widetilde V^{\star} \cap \mathbb Z^{2r-1}} \widetilde \ell_{\vec k}(n_0) h\bigg(\frac{\vec k}{N} \bigg)+O\left( \frac{N^d}{\log \log x}+\frac{N^d}{M^{10}} \right),
\end{split}
\]
which establishes an analogue of \eqref{eq:poisson}. To obtain the above estimate we have also used an analogue of Lemma \ref{lem:errorbd2} where $\widetilde \ell_{\vec k}(n_0)$ replaces $\ell_{\vec k}(n_0)$ in the left-hand side of \eqref{eq:maintermcancellation}. Since this result follows from a completely analogous argument to the one used to establish Lemma \ref{lem:errorbd2} we will omit the details.
Recalling Lemma \ref{lem:alpha-bound}, reversing our decomposition completes the proof.\end{proof}



\section{Proof of Theorem \ref{thm:corr2}: Matching the random model with the\\ $r$-correlation}

The goal of this section is to express our formulas for the
$r$-correlation of lattice points in terms of the random model from
Section \ref{sec:random-model} and complete the proof of Theorem
\ref{thm:corr2}. Our approach is to compute the smoothed
$r$-correlation of the random model $\mathcal R_r(F_N)$ as given in
\eqref{eq:randomcorr} by following a similar strategy to the one used
to compute the $r$-correlation for the angles of lattice points. Using
independence of the random variables $(\vartheta_l)_{l=1}^M$ we will
quickly arrive at the same expression that appears in Proposition
\ref{prop:corr}.   

\begin{proof}[Proof of Theorem \ref{thm:corr2}]
Recall the definition of $x_J$ as given in \eqref{eq:xJ-def}. Let us define the random variable
\begin{equation}\label{eq:random-coeff}
\lambda_{4k}=\sum_{J \subset \{1, \ldots, M\}} e^{4ki x_J}=\prod_{l=1}^M(1+e^{4ki \vartheta_{l}}
).
\end{equation}


We will first establish \eqref{eq:randomeq1}.  Using \eqref{eq:Fform},
and with $\ell_{\vec k}(n_0)$ as in \eqref{eq:elldef}, we have that
\begin{equation} \label{eq:randompoisson}
\mathcal R_{r,n_0}(F_N)=\frac{1}{N^r} \sum_{\vec k \in \mathbb Z^{r-1}} \ell_{\vec k}(n_0) \widehat f \bigg( \frac{\vec k}{N}\bigg) \prod_{j=0}^{r-1} \lambda_{4(k_{j+1}-k_j)},
\end{equation}
where we recall that by convention $k_0=k_r=0$. Using \eqref{eq:random-coeff} as well as that $(\vartheta_l)_{l=1}^M$ are independent and arguing as in \eqref{eq:combin} we have
\begin{equation} \label{eq:expectation-comp}
\begin{split}
\mathbb E \bigg(\prod_{j=0}^{r-1} \lambda_{4(k_{j+1}-k_j)} \bigg)=\prod_{l=1}^M \mathbb E\bigg( \prod_{j=0}^{r-1} (1+ e^{i4(k_{j+1}-k_j)\vartheta_l} \bigg) \\
=\prod_{l=1}^M \sum_{\mathbf S \subset \{0,1,\ldots, r-1\}} \mathbb E\big( e^{i4k_{\mathbf S}\vartheta_l}\big) =(2\alpha(\vec k))^{M}.
\end{split}
\end{equation}
Hence, combining this with \eqref{eq:randompoisson} gives
\[
\mathbb E\big(\mathcal R_{r,n_0}(F_N) \big)=\frac{1}{N^r} \sum_{\vec k \in \mathbb Z^{r-1}}  \widehat f \bigg( \frac{\vec k}{N}\bigg) \ell_{\vec k}(n_0) (2\alpha(\vec k))^M.
\]
Using this along with \eqref{eq:correxp} establishes \eqref{eq:randomeq1}.

It remains to prove \eqref{eq:randomeq2}. The argument proceeds similarly. Applying \eqref{eq:Fform} and arguing as in the first step of the proof of Proposition \ref{prop:var} yields
\begin{equation} \label{eq:randompoisson-var}
\mathcal R_{r,n_0}(F_N)^2=\frac{1}{N^{2r}} \sum_{\substack{\vec k \in \mathbb Z^{2r-1} \\ k_r=0}} h \bigg( \frac{\vec k}{N}\bigg) \widetilde \ell_{\vec k}(n_0)  \prod_{j=0}^{2r-1} \lambda_{4(k_{j+1}-k_j)},
\end{equation}
where $h$ is as in \eqref{eq:hdef}, $\widetilde \ell_{\vec k}(n_0)$ is as in \eqref{eq:ltildedef}, and we use the convention $k_0=k_{2r}=0$. Hence, it follows from \eqref{eq:expectation-comp} (which we use with $2r$ in place of $r$) that
\[
\mathbb E \big(\mathcal R_{r,n_0}(F_N)^2 \big)=\frac{1}{N^{2r}} \sum_{\substack{\vec k \in \mathbb Z^{2r-1} \\ k_r=0}} h \bigg( \frac{\vec k}{N}\bigg) \widetilde \ell_{\vec k}(n_0) (2\alpha(\vec k))^M.
\]
Therefore, using this together with \eqref{eq:var} establishes \eqref{eq:randomeq2}, which completes the proof.
\end{proof}


\section{Poisson correlations for the random model}
\label{sec:poiss-corr-hybr}
We first treat the square free case separately as it is notationally
simpler, and then use it to deduce the case for general $n$.
\subsection{The square free case}
\label{subsec:square free-cases-random-model}

We begin with the case of square free $n$, i.e.,
$n_{0} \in \{1,2\}$. 
For simpler notation, let $N = N_{n} = r(n)$.  In the random model we
can directly handle summing over distinct angles, and will work with the following setup for
the ``standard'' $r$-level correlation,
which we define as the random
variable
\begin{equation*}
  \mathcal{R}_{r}^*(\FN)
  :=
  \frac{1}{N}
\sumstar_{J_{1}, \ldots, J_{r} \subset \{1, \ldots, M\} }
\FN(x_{J_{1}}-x_{J_2}, \ldots, x_{J_{r-1}}-x_{J_{r}})
\end{equation*}
where $\sum^{*}$ indicates summing over {\em distinct} subsets $J_{1},
\ldots, J_{r}$. Throughout this section we assume $f$ has compact support.

\subsubsection{The pair correlation}
To illustrate ideas we begin by determining the pair correlation.
Our approach is to compute the expected value, and then, via a
variance bound, show that fluctuations around the mean are small.

{\em The expectation of $\mathcal{R}_2^*(\FN)$:}
Using linearity of expectations,  the expected value of
$\mathcal{R}_{2}^*(\FN)$ is given by
$$
\mathbb{E} 
\left(
  \mathcal{R}_{2}^*(\FN)
\right)  
=  
\frac{1}{N}
\sumstar_{J_{1},J_{2} \subset \{1, \ldots, M\}}
\mathbb{E} 
\left(
\FN(x_{J_{1}}-x_{J_{2}})
\right).
$$
Since $\sum^{*}$ indicates summing over distinct subsets
$J_{1},J_{2}$, the symmetric difference
$J_{1} \Delta J_{2}$ is nontrivial and $x_{J_{1}}-x_{J_{2}}$ is a sum
(with certain choices of signs) of $|J_{1} \triangle J_{2}| \geq 1$
{\em independent} uniform random variables on the torus (here and in
what follows $J_{1} \triangle J_{2}$ denotes the symmetric difference
between the sets $J_{1}$ and $J_{2}$.)
In particular, for distinct subsets $J_{1},J_{2}$ we have by a direct
computation that 
$\mathbb{E} \left( \FN(x_{J_{1}}-x_{J_{2}}) \right) = \hat{f}(0)/N$,
and the total contribution equals
$$
\frac{1}{N} \cdot N(N-1) \frac{\hat{f}(0)}{N} =
\hat{f}(0)(1-1/N) = 
\hat{f}(0) + O_{f}(1/N).
$$

{\em Bounding the variance of $\mathcal{R}_2(\FN)$:}
We next show that the fluctuations around the mean, with large
probability, are small in comparison with the mean, namely that
$$
\E( \mathcal{R}_{2}^*(F_N)^{2} ) -  \E(\mathcal{R}_{2}^*(F_N))^{2} = o(1)
$$
as $M$ tends to infinity. By linearity of expectations, we find that 
\begin{equation}
  \label{eq:pair-variance}
\E( \mathcal{R}_{2}^*(F_N)^{2} )
=
\frac{1}{N^{2}}
\sumstar_{J_{1},J_{2} \subset \{1, \ldots, M\}} \,
\sumstar_{J_{3},J_{4} \subset \{1, \ldots, M\}}
\mathbb{E} 
\left(
\FN (x_{J_{1}}-x_{J_{2}})
\FN  (x_{J_{3}}-x_{J_{4}})
\right).
\end{equation}
Now, for a generic choice of subsets (i.e., for $N^{4}(1+o(1))$
choices), the two components of
$(x_{J_{1}}-x_{J_{2}},x_{J_{3}}-x_{J_{4}})$ will contain at least one
pair of independent random variables, and if this is so  we have
$$
\mathbb{E} 
\left(
\FN (x_{J_{1}}-x_{J_{2}})
\FN (x_{J_{3}}-x_{J_{4}})
\right)
= \frac{\hat{f}(0)^{2}}{N^{2}}
$$
and hence the main term of (\ref{eq:pair-variance}) equals
$$
\frac{N^{4}(1+o(1))}{N^{2}} \cdot \frac{\hat{f}(0)^{2}}{N^{2}}
=
\hat{f}(0)^{2}
+ o_{f}(1).
$$

\begin{rem}
  A delicate issue is that if we would include pairs of non-distinct
  subsets --- we then get ``degenerate pairings'' which give a
  contribution of the same size as the main term. Namely, if we take
  $J_{1}=J_{2}$ and $J_{3}=J_{4}$ (there are $N^{2}$ such choices), we
  find, on noting that
  $\E( \FN(x_{J_{1}}-x_{J_{2}}) \cdot \FN(x_{J_{3}}-x_{J_{4}})) =
  f(0)^{2}$, that the contribution from these terms equals
$$
\frac{1}{N^{2}} \cdot N^{2}f(0)^{2} = f(0)^{2}.
$$
In fact, the same holds for the expectation! However, this should not
be a surprise as allowing for pairs of points to be equal should give
a secondary main term of the form $f(0)$ in addition to $\widehat{f}(0)$.
\end{rem}

To bound the degenerate solutions we argue as follows: each choice of
subsets $J_{1},J_{2},J_{3},J_{4}$ gives a group homomorphism
$$
\T^{M} \to \T \times \T, \quad
(\vartheta_1, \ldots, \vartheta_M) \to (
\sum_{j \in J_{1}} \vartheta_{j} - \sum_{j \in J_{2}} \vartheta_{j},
\sum_{j \in J_{3}} \vartheta_{j} - \sum_{j \in J_{4}} \vartheta_{j}
),
$$
which, on letting $\vec \vartheta = (\vartheta_1, \ldots, \vartheta_M)$, can be
written as 
$$
\vec \vartheta \to ( (\vec v_{1}-\vec v_{2}) \cdot \vec \vartheta, (\vec v_{3}-\vec v_{4}) \cdot \vec \vartheta)
$$
where $\vec v_{l} = \sum_{i \in J_{l}} \vec e_{i}$ and $\vec
e_{1},\ldots,\vec e_{M}$ 
denotes the standard basis for $\R^{M}$. (Note that the torus map is
given by an integer entry matrix in
$\operatorname{Mat}_{2,M}(\Z)$
whose rows are given by $(\vec v_1-\vec v_2)^{t},(\vec v_3- \vec v_4)^{t}$,
and if this matrix has rank two the distribution of the image inside
$\T^2$ is uniform; for a formal argument see
Lemma~\ref{lem:uniform-push-forward}.) We note that if the $4$-tuple
of subsets $(J_{1}, J_{2},J_{3},J_{4})$ is ``generic'', then the rank of this
map is $2$; when this does not hold we call the $4$-tuple of subsets
$(J_{1}, J_{2},J_{3},J_{4})$ ``degenerate''.  If the rank is zero, we
must have $\vec v_{1}=\vec v_{2}$ and $\vec v_{3}= \vec v_{4}$, hence
$J_{1}=J_{2}$ and $J_{3}=J_{4}$ and thus there is no contribution, and
similarly there is no contribution if the rank is one due to either
$\vec v_{1}-\vec v_{2}=0$ or $\vec v_{3}-\vec v_{4}=0$. The remaining
rank one case is that there exists nonzero scalars $\alpha,\beta$ such
that
$$
\alpha(\vec v_{1}-\vec v_{2}) = \beta(\vec v_{3}-\vec v_{4})
$$
where, we may without loss of generality assume that $\alpha,\beta$
are coprime integers, and say $\alpha >0$. We note that $\alpha > 1$
is impossible since the components of the vector $\alpha(\vec v_{1}-\vec v_{2})$
are in $\{-\alpha,0,\alpha \}$, whereas the components of
$\beta(\vec v_{3}-\vec v_{4})$ are in $\{-\beta,0,\beta \}$ (note that
at least 
one component in each of the two vector differences must be nonzero.)
The remaining case is that $\alpha = 1$ and $\beta = \pm 1$; say
$\beta = 1$ (the other case follows similarly.)  In this case, for
$\vec v_{1},\vec v_{2}$ fixed, we find that
$J_{3}\setminus J_{4} = J_{1} \setminus J_{2}$, as well as
$J_{4}\setminus J_{3} = J_{2} \setminus J_{1}$, and the only choice
left is specifying the intersection $J_{3} \cap J_{4}$, which clearly
must be contained in the {\em complement} of the symmetric difference
$J_{1} \triangle J_{2}$. We next show that the cardinality of the
symmetric difference, for ``generic'' choices of
$J_{1},J_{2}$, is
$M\cdot (1/2 + o(1))$.
\begin{lem}
  \label{lem:law-large-numbers}
For $N^{2}(1+o(1))$ choices of subsets $J_{1},J_{2} \subset \{ 1,\ldots, M \}$
we have $|J_{1} \triangle J_{2}| = M \cdot (1/2 + o(1))$, as $M \to \infty$.
\end{lem}
\begin{proof}
  We use the following simple probabilistic argument: first note that the number of
  pairs of subsets with the desired property is $N^{2}$ times the
  probability of randomly selected subsets $J_1,J_2$ having the same
  property, where the two subsets are selected independently and the
  probability for each configuration is $1/2^{M}$. Or equivalently,
  each element $i \in \{1, \ldots, M\}$ is independently selected to
  be in $J_{1}$ with probability $1/2$, and similarly for
  $i \in J_{2}$. In particular, for a fixed index $i$, each of the
  four possible containment patterns w.r.t. $J_{1},J_{2}$ occurs with
  probability $1/4$.  Hence an index $i \in \{1, \ldots, M\}$ is
  contained in $J_{1} \triangle J_{2}$ with probability $1/2$ (as this
  occurs for two out of the four possible containment patterns.)
  Since the events for different indices $i$ are independent, by the
  weak law of large numbers we find that
  $|J_{1} \triangle J_{2}|/M = 1/2 + o(1)$ holds with probability
  $1+o(1)$ as $M$ grows, and the result follows.
\end{proof}
The Lemma immediately gives that for
all $N^{2}(1+o(1))$ ``generic'' choices of  $J_{1},J_{2}$ (also distinct), the
symmetric difference has size roughly of order $M/2$, hence leaving
$2^{M-|J_{1} \triangle J_{2}|} = o(N)$ possibilities to choose the
intersection $J_{3} \cap J_{4}$. The total contribution is thus
$$
\ll_{f}
\frac{1}{N^{2}} (N^{2} \cdot o_{f}(N))   \frac{1 }{N} = o_{f}(1)
$$
(here we have used that
$ \E (|\FN(x_{J_1}-x_{J_2}) \cdot \FN(x_{J_3}-x_{J_4})| )
=
\E (|\FN(x_{J_1}-x_{J_2})^{2}| 
\ll_{f} 1/N$; it is
crucial to use the fact that the rank is one.)

Finally, for the $o(N^{2})$ non-generic choices of $J_{1},J_{2}$, there
can be at most $N=2^{M}$ possibilities for the intersection $J_{3}
\cap J_{4}$; again the total contribution is
$$
\ll_{f} o(N^{2}) \cdot N \cdot \frac{ 1}{N^{2}} \cdot \frac{1}{N}
=
o_{f}(1).
$$

\subsubsection{Higher level correlations}
\label{sec:high-level-corr}
The general case is more involved combinatorially, but the key idea is
still that the expectation of the $r$-level correlation for
``generic'' $x_{J_1}, \ldots, x_{J_r}$ (i.e., generic choices of subsets
$J_{1},\ldots, J_{r}$)
dominates, and then to bound degenerate tuples.

{\em The expectation of $\mathcal{R}_{r}^*$}:
We begin with a  result regarding uniform distribution on tori.
\begin{lem}
\label{lem:uniform-push-forward}
Let $A \in \tmop{GL}_{n}(\Q) \cap \tmop{Mat}_{n \times n}(\Z)$, and let
$X = (X_{1},\ldots,X_{n})$ denote a uniformly distributed random
variable on $\T^n$.  Then $Y = AX$ is also uniform.  More generally,
if $n \ge m$ and $A \in \tmop{Mat}_{m \times n}(\Z)$ is an
$m \times n$-matrix with rank $m$, and $X$ is a uniformly distributed
random variable on $\T^n$, then $Y= AX$ is uniformly distributed on
$\T^m$.
\end{lem}
\begin{proof}
  For $n=1$, i.e., $Y = Y_{1} = a X_{1}$, for nonzero $a \in \Z$ the
  result is clear (just consider the preimage of a small interval; it
  will consist of $|a|$ copies of intervals whose lengths are scaled
  by $1/|a|$.)  For $n>1$, $A$ can be decomposed as
  $A = B_{1} D B_{2}$ where $B_{1},B_{2} \in \tmop{SL}_{n}(\Z)$ and $D$ is a
  diagonal matrix with integer entries (we can take $D$ to be the
  Smith normal form $A$, cf.  \cite[Theorem
  2.4.12]{cohen-comp-alg-nt-book}).  It is thus enough to prove the
  statement for 
  $A \in \tmop{SL}_{n}(\Z)$ or $A=D$.  The former is clear as multiplication
  by $A$ does not change the measure (the determinant of the Jacobian
  equals $\pm 1$.)  The case $A=D$ and $n=1$ is already done, and
  writing $dx=dx_{1} \cdots d x_{n}$ the general case follows by
  change of measure one component at a time.

  For $A \in \tmop{Mat}_{m \times n}(\Z)$ the argument is similar.  After
  permuting columns we may assume that the $m$ first columns of $A$
  are linearly independent.  Decomposing the uniform measure on
  $\T^{n}$ as a product of two uniform measures $\mu_{1} \times \mu_2$
  on $\T^{m} \times \T^{n-m}$, the result follows by conditioning on
  the $\T^{n-m}$-component, using the first part applied to
  $\T^{m}$-component, together with the uniform measure being
  translation invariant.
\end{proof}

Before proceeding we introduce some further notation.
As before, given a subset $J_{i} \subset \{1, \ldots, M\}$, let
$x_{i} = \sum_{j \in J_{i}} \vartheta_{j}$.  In order to discuss linear
independence and rank, first recall that 
$\vec v_{i} =\sum_{j \in J_{i}} \vec e_{j}$ where
$\vec e_{1}, \ldots, \vec e_{M}$ denotes the standard basis of
$\R^M$.  We can then write
$$
x_{i} = \vec v_{i} \cdot \vec \vartheta
$$
where $\vec \vartheta = (\vartheta_{1}, \ldots, \vartheta_{M})$ denotes a vector of
$M$ uniform and independent random variables taking values in $\T^1$.
Letting
$$
\Delta_{i} := x_{i}-x_{i+1}, \quad \text{ and } \quad  \vec w_{i} := \vec
v_{i}-\vec v_{i+1}, 
\quad i=1,\ldots,r-1
$$
we find that
$$
\Delta_{i} = (\vec v_{i}-\vec v_{i+1}) \cdot  \vec \vartheta = \vec w_{i}
\cdot \vec \vartheta. 
$$
\begin{lem}
  \label{lem:expectation}
Let $f$ be a compactly supported Schwartz function.
Given any $r$-tuple of distinct subsets $J_{1}, \ldots, J_{r} \subset \{1,
\ldots, M\}$ with associated difference vectors 
$\vec w_{1}, \ldots, \vec w_{r-1}$, let  $d$ denote the dimension of the
 vector space spanned by these vectors.  If $d$ is
maximal, i.e., $d=r-1$, then
$$
\E
\left(
  \FN(x_{1}-x_{2}, \ldots, x_{r-1}-x_{r})
\right)
=
\frac{1}{N^{r-1}} \widehat{f}(0).
$$
If $0 < d < r-1$, then
$$
\E
\left(
|\FN(x_{1}-x_{2}, \ldots, x_{r-1}-x_{r})|
\right)
\ll_{f} 1/N^{d}.
$$
\end{lem}
\begin{proof}
  The first part is an immediate consequence of
  Lemma~\ref{lem:uniform-push-forward}, as
  $(\Delta_{1}, \ldots, \Delta_{r-1})$ is a random variable uniformly
  distributed on $\T^{r-1}$.

  The second part follows by a similar argument: choose $d$ independent 
  vectors $\vec w_{j_{1}}, \vec w_{j_{2}}, \ldots \vec w_{j_{d}}$, and use that
  \begin{equation}
    \label{eq:expectation-sandwich}
\E(     |\FN(x_{1}-x_{2}, \ldots, x_{r-1}-x_{r} )|)
\ll_{f}
\E( g_{N}(x_{j_{1}}-x_{j_{2}}, \ldots, x_{j_{d}}-x_{j_{d+1}}) )
\ll_{f} 1/N^{d},
\end{equation}
where $g_{N}(x_{j_{1}}-x_{j_{2}}, \ldots, x_{j_{d}}-x_{j_{d+1}})$ is
obtained from $\FN$ by taking the supremum of
$|\FN(x_{1}-x_{2}, \ldots, x_{r-1}-x_{r})|$ over the coordinates
corresponding to $x_{j_{i}}-x_{j_{i+1}}$ fixed for $i=1,\ldots,d$ and
the 
other ones ranging freely.
(Here we use that the support
of $g_{N}$ is 
contained in some $d$-dimensional ball of radius $\ll_{f} 1/N$.)
\end{proof}

Before proceeding we next show that all low rank subspaces can be
defined via $O_{r}(1)$ linear forms (i.e., it does not depend on
$M$.)
\begin{lem}
  \label{lem:bounded-number-of-W}
  Let $W = \operatorname{Span}(\vec w_{1}, \ldots, \vec w_{r-1})$ and
  assume that $d := \dim(W) < r-1$. Choosing a basis for $W$
  consisting of $d$ independent elements of
  $\{\vec w_{1}, \ldots, \vec w_{r-1} \}$, the remaining vectors are
  given by $r-1-d$ linear forms in the basis vectors, and the number of
  distinct collections of such forms is then $O_{r}(1)$ (in particular, the
  estimate is uniform for all tuples $J_{1}, \ldots, J_{r-1}$ as long
  as $d < r-1$.)
\end{lem}
\begin{proof}
  After renumbering indices (there are at most $(r-1)!=O_{r}(1)$ ways to do
  this) we may assume that $\vec w_{1}, \ldots, \vec w_{d}$ are independent, and
  that for all integers $i \in [1, r-1-d]$ we have
$$
\gamma_{i} \vec w_{d+i} = L_{i}(\vec w_{1}, \ldots, \vec w_{d})
$$
with $\gamma_{i} \neq 0$ an integer and each $L_{i}$ a linear form
with integer coefficients, with the property that the gcd of
$\gamma_{i}$ and the coefficients of $L_{i}$ equals one.  Form a
$M \times (d+1)$ matrix $A$ having columns
$\vec w_{1}, \ldots, \vec w_{d}, \vec w_{d+i}$; the above linear
relation can then be 
formulated in terms of the existence of a nonzero vector
$\vec l \in \Z^{d+1}$ so that $A \vec l = 0$.  Since $\vec w_{1},
\ldots , \vec w_{d}$ are 
independent, we may form a $d \times (d+1)$ matrix $B$, having rank
$d$, by selecting $d$ independent rows in $A$, with the property that
$A \vec l = 0$ if and only if $B \vec l = 0$. Further the set of
$\{ \vec l : B \vec l = 0 \} $ lies on a line, so if the coordinates of $\vec l$ have gcd one (in analogy with the above gcd condition), $\vec l$ is up to
sign uniquely determined by $A$.  On the other hand, each entry in $B$
lies in $\{-1,0,1\}$, hence there are at most $O_{d}(1) = O_{r}(1)$
possible ways to choose $B$, and hence the number of ways to
choose $\vec l$ is also $O_{r}(1)$.

Thus, for each $i$, there are $O_{r}(1)$   ways to choose $\gamma_{i}, L_{i}$,
and thus there are in total $O_{r}(1)$ possible ways
to select $r-1-d$ such linear relations.
\end{proof}

\begin{prop}
\label{prop:expectation-counting-tuples}  
  The number of $(\vec v_{1},\vec v_{2}, \ldots, \vec v_{r}) \in
  (\{0,1\}^{M})^{r}$ 
  such that $\vec v_{i} \neq \vec v_{j}$ for $i\neq j$, and so that the rank of
  $(\vec w_{1}, \ldots, \vec w_{r-1})$ equals $r-1$, is 
$$
N^{r}(1+o_{r}(1))
$$
as $M \to \infty$.

Further, the number of
$(\vec v_{1},\vec v_{2}, \ldots, \vec v_{r}) \in (\{0,1\}^{M})^{r}$
that are pairwise 
distinct, and such that the rank of $(\vec w_{1}, \ldots, \vec w_{r-1})$ equals
$d < r-1$ is
$$
o_{r}(N^{d+1}),
$$
as $M \to \infty$.

\end{prop}

Before giving the proof we record the following simple linear algebra
result.
\begin{lem}
  \label{lem:consecutive-w}
  Let $W = \operatorname{Span}(\vec w_{1},\ldots, \vec w_{r-1})$ and let
  $d = \dim(W) < r-1$.  Then there exists a permutation of indices
  $i_{1},i_{2}, \ldots, i_{r}$ such that if we let
  $\vec w_{j}' := \vec v_{i_{j}}-\vec v_{i_{j+i}}$ for $j=1, \ldots,
  r-1$, we have 
  $W = \operatorname{Span}(\vec w_{1}',\ldots, \vec w_{d}')$.

\end{lem}
\begin{proof}
  We begin by noting that the vectors $\vec w_{1}, \ldots, \vec w_{r-1}$ have the
  same span as the vectors
  $\vec v_{1}-\vec v_{2}, \vec v_{1}-\vec v_{3}, \ldots, \vec
  v_{1}-\vec v_{r}$, since they are 
  related by a lower triangular matrix all whose entries are one.
  Thus there exist a relabeling of indices $i_{1}, \ldots, i_{r}$ such
  that $i_{1} =1$, and that the vectors
  $\vec v_{i_{1}}-\vec v_{i_{2}}, \vec v_{i_{1}}-\vec v_{i_{3}},
  \ldots, \vec v_{i_{1}}-\vec v_{i_{d+1}}$ are 
  linearly independent. Letting $\vec w_{j}' =
  \vec v_{i_{j}}-\vec v_{i_{j+1}}$ for $j=1,\ldots, r-1$, and arguing as above, 
  we find that 
  that $\vec w_{1}',\vec w_{2}',\ldots, \vec w_{d}'$ are independent.
\end{proof}

\begin{proof}[Proof of Proposition \ref{prop:expectation-counting-tuples}]
  We give a simple probabilistic proof showing that certain events
  occur with probability $1+o(1)$ as $M$ grows.  Namely, pick vectors
  $\vec v_{1}, \ldots, \vec v_{r} \in \{0,1\}^{M}$ by fair and independent coin
  flips --- to later obtain asymptotics for counts of vectors with
  properties of interest we then multiply said probability with
  $2^{Mr}=N^{r}$.  For $M$ large, the vectors are all pairwise
  distinct with probability $1+o(1)$.  Further, if we consider a
  coordinate $j \in \{ 1, \ldots, M\}$ and fix $i \in \{1,\ldots,r\}$,
  the likelihood that the $j$-th coordinate of $\vec v_{i}$ is one, and
  that the $j$-th coordinate for all other vectors $\vec v_{l}$,
  $l\in \{1,\ldots,r\} \setminus \{ j \}$, is zero is $1/2^{r}$.
  Thus, by the law of large numbers (as in the proof of
  Lemma~\ref{lem:law-large-numbers}) with probability $1+o(1)$ as
  $M \to \infty$, there are $r$ 
  indices $j_{1},j_{2}, \ldots,j_{r}$ so that
  $\vec v_{l} \cdot \vec e_{j_{i}} = \delta_{li}$ for $l =
  1,\dots,r$. Consequently the rank of $\vec w_{1},\ldots,\vec w_{r-1}$ is $r-1$
  with probability $1+o_{r}(1)$, which implies the first part.
  
  For the second part, we use Lemma~\ref{lem:consecutive-w} to reduce
  to the case of $\vec w_{1},\ldots,\vec w_{d}$ being independent, and that
  $\vec w_{d+1}$ can be written as a linear combination of the first $d$
  vectors. In particular there exists scalars $\alpha_{i}$ such that
$
\sum_{i=1}^{d+1} \alpha_{i} \vec w_{i} = \vec 0
$
  (with $\alpha_{d+1} \neq 0$),  
  leading to a linear relation
  $$
\sum_{i=1}^{d+2} \beta_{i} \vec v_{i} = \vec{0}
  $$
  where $\beta_{1} = \alpha_{1}, \beta_{d+2} = - \alpha_{d+1}$, and
  otherwise $\beta_{j} = \alpha_{j+1}-\alpha_{j}$; note in particular
  that $\beta_{d+2} \neq 0$.  Fix the coefficients of such a relation
  (note that there are $O_{r}(1)$ possible relations by
  Lemma~\ref{lem:bounded-number-of-W}.)
  Further, if $l$ denotes the smallest
  integer such that $\alpha_{l} \neq 0$, we also find that
  $\beta_{l} \neq 0$, in particular there are at least two
  nonvanishing $\beta_{i}'s$.
  If there are exactly two nonzero $\beta_{i}$'s, we have
  $\beta_{l}\vec v_{l}= - \beta_{d+2}\vec v_{d+2}$.  Since
  $\beta_{l},\beta_{d+2} \neq 0$ (and at least one of $\vec v_{l}$ and
  $\vec v_{d+1}$ must be nonzero), we in fact have have $\vec v_{l},
  \vec v_{d+2} \neq \vec 0$, 
and since the coordinates of both $\vec v_{l}$ and $\vec v_{d+2}$ are in
$\{0,1\}$ we must have $\beta_l = -\beta_{d+2}$, and thus
$\vec v_{l} = \vec v_{d+2}$, contradicting the vectors being assumed to be
pairwise distinct.

Now, if $\beta_{i}   \neq 0$ for at least three values, we may write
$$
\alpha \vec v_{i_{1}} + \gamma \vec v_{i_{2}} = \sum_{\substack{1 \le i \le d+2 \\i \neq i_{1}, i_{2}}}
\delta_{i} \vec v_{i}
$$
with $\alpha,\gamma >0$.  In particular, the support\footnote{By the
  support of a vector we mean the set of indices for which the
  corresponding coordinates are nonzero.} of
$\vec v_{i_{1}},\vec v_{i_{2}}$ is determined by the right-hand side, and thus
(after taking the relabeling of indices into account), so is
$J_{i_{1}} \cup J_{i_{2}}$.  In fact, $J_{i_{1}} \cap J_{i_{2}}$ as
well as the symmetric difference $J_{i_{1}} \Delta J_{i_{2}}$ is also
determined by the right-hand side.  Letting $M_{0}$ denote the cardinality of said
determined symmetric difference, we find that there are $2^{M_{0}}$
ways to choose $J_{i_{1}}, J_{i_{2}}$ for a fixed right-hand side (note that there
are at most $N^{d}$ possibilities for the right-hand side).  Now, for a
``generic'' right-hand side, a simple probabilistic argument (i.e., using the law
of large numbers as before) to count the number of indices not in any
$J_{i}$ for $i \neq i_{1},i_{2}$, shows that for $N^{d}(1+o_{r}(1))$
of the possible choices of the right-hand side, we have
$|\cup_{i \neq i_{1},i_{2}} J_{i}| \leq (1-\epsilon)M$ (where
$\epsilon = \epsilon_{r}>0$ is allowed to depend on $r$ but not on
$M$.)  If this is the case, we have $M_{0} \leq (1-\epsilon)M$ and
there are at most $2^{(1-\epsilon)M}=o_{r}(N)$ ways to choose
$J_{i_{1}},J_{i_{2}}$, for a total of $o_{r}(N^{d+1})$ possibilities.
On the other hand, if the right-hand side is one of the $o_{r}(N^{d})$
``non-generic'' choices (in particular allowing $M_{0}=M$), there are
$2^{M_{0}} \leq 2^{M} = N$ possible ways to choose
$J_{i_{1}},J_{i_{2}}$, and the total number of possibilities is also
here $o_{r}(N^{d+1}$).

Finally, once $\vec v_{1}, \ldots, \vec v_{d+2}$ are chosen, {\em all}
$\vec w_{i}$ are determined since
$\vec w_{d+1}, \vec w_{d+2}, \ldots, \vec w_{r-1}$ depend linearly on
$\vec w_{1}, \ldots, \vec w_{d}$, and thus the vectors
$\vec v_{d+3}, \ldots \vec v_{r}$ are also uniquely determined.  As the
number of linear relations is $O_{r}(1)$, we find that total number of
ways to choose $\vec v_{1}, \ldots, \vec v_{r}$ so that
$\vec w_{1}, \ldots, \vec w_{r-1}$ has non-maximal rank $d<r-1$ is
$o_{r}(N^{d+1})$.
\end{proof}

\begin{thm}
\label{thm:random-expectation}
Let $f$ be a compactly supported Schwartz function.
As $M$ grows we have
  $$
\E 
\left(
\mathcal{R}_{r}^*(\FN)   
\right)
= \widehat{f}(0) +o_{f}(1)
$$
and
$$
\E 
\left(
\mathcal{R}_{r}^*(\FN)^2   
\right)
= \widehat{f}(0)^{2}
 +o_{f}(1).
$$
\end{thm}
\begin{proof}
{\em The expectation:}  
By linearity of the expectation, we have
  $$
\E 
\left(
\mathcal{R}_{r}^*(\FN)   
\right)
=
\frac{1}{N} \sumstar_{J_{1},\ldots, J_{r} \subset \{1,\ldots,M\}}
\E 
\left(
\FN(x_{1}  -x_{2}, \ldots, x_{r-1}-x_{r})
\right).
  $$
  The main term arises from ``generic'' choices of distinct subsets
  $J_{1}, \ldots, J_{r-1}$ (i.e., the full rank ones); by
  Proposition~\ref{prop:expectation-counting-tuples} there are
  $N^{r}(1+o_{r}(1))$ such tuples, and by
  Lemma~\ref{lem:expectation}, each such term
    contributes
  $$
\frac{1}{N^{r-1}}  \widehat{f}(0).
$$

The contribution from the ``non-generic'' choices of distinct subsets (i.e.,
having rank $d < r-1$), is then (after taking into account there being
at most $O_{r}(1)$ possible linear relations), using the
second parts of 
  Proposition~\ref{prop:expectation-counting-tuples} 
and Lemma~\ref{lem:expectation}, is
$$
\ll_{r,f} \frac{1}{N}
\sum_{d=1}^{r-2} \frac{1}{N^{d}}
o_{r}(N^{d+1})
= o_{r,f}(1).
$$

{\em Bounding the variance:}
The argument to bound the variance is similar to the one used to
bound the pair correlation variance. We first note that
\begin{multline}
\E 
\left(
\mathcal{R}_{r}^*(\FN) ^2  
\right)=
\frac{1}{N^{2}} \sumstar_{J_{1},\ldots, J_{r} \subset \{1,\ldots,M\}}
  \\
\sumstar_{J_{1}',\ldots, J_{r}' \subset \{1,\ldots,M\}}
\E 
\left(
  \FN(x_{1}-x_{2}, \ldots, x_{r-1}-x_{r}) \cdot
    \FN(x_{1}'-x_{2}', \ldots, x_{r-1}'-x_{r}')
  \right).
\end{multline}
  The number of ``generic'' tuples (i.e., giving the full rank
  $2(r-1)$) is $N^{2r}\cdot (1+o(1))$, and for these we have
  $$
\E 
\left(
  \FN(x_{1}-x_{2}, \ldots,x_{r-1}-x_r) \cdot
    \FN(x_{1}'-x_{2}', \ldots, x_{r-1}'-x_r') 
\right) = 
\left(
  \frac{\widehat{f}(0)}{N^{r-1}}
\right)^{2}.
$$

To count the contribution from ``non-generic'' tuples with rank
$d < 2(r-1)$ we argue as follows.  Letting $W$ denote the linear span
of $\vec w_{1}, \ldots, \vec w_{r-1}, \vec w_{1}', \ldots, \vec w_{r-1}'$ we have
$d = \dim(W)$, and by an argument similar to the one used in
Lemma~\ref{lem:expectation} (i.e., use an inequality analogous to
(\ref{eq:expectation-sandwich}) together with
Lemma~\ref{lem:uniform-push-forward}) gives that
  $$
\E 
\left(
| \FN(x_{1}-x_{2}, \ldots, x_{r-1}-x_r) \cdot
    \FN(x_{1}'-x_{2}', \ldots, x_{r-1}'-x_r')  | 
\right)
\ll_{f} 1/N^{d}.
$$

It thus suffices to show that the number of ``non-generic'' tuples is
$o(N^{d+2})$.  We may after relabeling the $\vec v_{i}$ as well as the
$\vec v_{i}'$-indices select $d_{1}$ vectors
$\vec w_{1}, \ldots, \vec w_{d_{1}}$ and $d_{2}$ vectors
$\vec w_{1}',\ldots, \vec w_{d_{2}}'$ (with the vectors $\vec w_{i}$
and $\vec w_{i}'$ defined using the reordered indices)
$d = d_{1}+d_{2}$ and
  $$
  \operatorname{Span}(\vec w_{1}, \ldots, \vec w_{d_{1}}, \vec w_{1}',\ldots,
\vec w_{d_{2}}')
= 
\operatorname{Span}(\vec w_{1}, \ldots, \vec w_{r-1}, \vec
w_{1}',\ldots, \vec w_{r-1}'). 
$$
More precisely, relabel both the $\vec v_{i}$ and the $\vec v_{j}'$ sets of
indices as in the proof of Lemma~\ref{lem:consecutive-w}, in
particular with $\vec v_{i}$ given by $J_{i}$ after relabeling the
first set of indices put $\vec w_{i} = \vec v_{i+1}-\vec v_{i}$, and
similarly with $\vec v_{i}'$ given by $J_{i}'$ after relabeling the
second set of indices put $\vec w_{i}' := \vec v_{i}'-\vec v_{i+1}'$.
We then find that a basis for $W$ is given by
$\vec w_{1},\ldots,\vec w_{d_{1}}, \vec w_{1}',\ldots, \vec
w_{d_{2}}'$.

Consider now a minimal non-trivial linear relation, i.e.,
$$
\sum_{i=1}^{d_1} \alpha_{i} \vec w_{i}
+
\sum_{i=1}^{d_2} \alpha_{i}' \vec w_{i}'
+
\gamma \vec w = 0
$$
where $\gamma \neq 0$ and $ \vec w= \vec w_{d_{1}+1}$ or
$\vec w=\vec w_{d_{2}+1}'$.  In case the
relation purely involves either $\vec w_{i}'$ or $\vec w_{i}$-vectors
the same argument used to treat the expectation suffices: if (say) the
relation only involves $w_{i}'$ vectors (so that
$\alpha_{1} =\ldots =\alpha_{d_{1}} = 0$ and
$\vec w=\vec w_{d_{2}+1}'$), the number of such tuples is
$$
\ll
N^{d_{1}+1} o_{r}(N^{d_{2}+1})
= o_{r}(N^{d+2})
$$
as there are at most $N^{d_{1}+1}$ choices for the
unconstrained $\vec v_{1},\ldots, \vec v_{d_{1}+1}$, and the remaining
$\vec v_{i}$ 
are determined by the first $d_{1}+1$ ones; 
by Proposition~\ref{prop:expectation-counting-tuples} 
there are $o(N^{d_{2}+1})$ possible $\vec v_{i}'$-tuples.
In case
the relation involves at least one $\vec w_{i}$-vector and at least one
$\vec w_{i}'$-vector,
we obtain a linear relation
$$
\sum_{i=1}^{d_1+1} \beta_{i} \vec v_{i} + \sum_{j=1}^{d_2+1} \beta_{j}' \vec v_{j}' = 0
$$
with at least three nonzero coefficients --- the same argument used to
prove Proposition~\ref{prop:expectation-counting-tuples}
then gives that the number of such $2r$-tuples is also
$o( N^{d_{1}+1 + d_{2}+1}) = o(N^{d+2})$.
\end{proof}

\subsection{Poisson correlations in the general case} 
We next consider 
integers $n_{0}$,
where $n_{0}$ is fixed and allowed to have prime power divisors, 
and 
the angle contribution from the $n_{0}$-part is entirely
deterministic, whereas we use the random model for the
square free part. We define
\begin{equation}
\label{eq:beta-shift}  
\mathcal R_{r,n_0}^*(F_N) 
:= 
\frac{1}{N}
\sumstar_{J_{1}, \ldots, J_{r} \subset \{1, \ldots, M\}}
\, \,
\sum_{(\beta_{1}), \ldots, (\beta_{r}) \subset \OO : |\OO/\beta_i|=n_{0}}
F_{N}( \theta_{\beta_{2}}-\theta_{\beta_{1}} +
  x_{J_{2}}-x_{J_{1}}, \ldots ).
\end{equation}
(Note
that with probability one the angles $x_{J_1}+\theta_{\beta_1}$ and
$x_{J_2}+\theta_{\beta_2}$ are equal if and only if $J_1=J_2$ and
$(\beta_1)=(\beta_2)$ so that with probability one the sum above is
over distinct points $x_{J_1}+\theta_{\beta_1},\ldots,
x_{J_r}+\theta_{\beta_r}$.) 

\begin{lem}
  \label{lem:hybrid-moments}
  Fix $n_{0}$ such that $N_{0} := r(n_{0})>0$.  
Let $f$ be a compactly supported Schwartz function.
Further, given an integer $M > 0$ let $N=2^{M}r(n_0)$.
Then, as $M$ grows,
$$
\mathbb E \big(  \mathcal{R}_{r,n_0}^*(F_N) \big)
=
\widehat{f}(0) + o_{f}(1),\qquad \text{and} \qquad
\mathbb E \big(   \mathcal{R}_{r,n_0}^*(F_N)^2 \big)
=
\left(  \widehat{f}(0)
\right)^{2}
+ o_{f}(1).
$$
\end{lem}
\begin{proof}
  If $n_{0}=1$ then
  Theorem~\ref{thm:random-expectation} immediately gives
  the result.

   For general $n_{0}>1$ we now consider the expectation of
   $\mathcal R_{r,n_0}^{*}(F_N)$.
%
First note that we, by definition, have
$J_{i} \neq J_{j}$ for all $1 \leq i \neq j \leq r$ for all terms of $\sum^{*}$
in (\ref{eq:beta-shift}).
With $N_{1} = 2^{M}$ and $N_0=r(n_0)$ we have $N= N_{0} \cdot N_{1}$;
and by the same argument used in Section~\ref{subsec:square
  free-cases-random-model} (the key point is that 
we may shift, and multiply by a bounded factor $1/N_{0}$
each coordinate of $F_{N}$)
we find that for each fixed $r$-tuple
$(\beta_{1}, \ldots, \beta_{r})$, we have
$$
\frac{1}{N}
\sumstar_{J_{1}, \ldots, J_{r} \subset \{1, \ldots, M\}}
\E
\left(
F_{N}( \theta_{\beta_{2}}-\theta_{\beta_{1}} +
  x_{J_{2}}-x_{J_{1}}, \ldots )
\right)
= (1/N_{0}^{r})  \left( \widehat{f}(0) + o_{f}(1) \right).
$$
Summing over the $N_{0}^{r}$
configurations of $r$-tuples $(\beta_{1}, \ldots, \beta_r)$ we find
that (\ref{eq:beta-shift}) equals $ \widehat{f}(0) + o_{f}(1)$.

The argument to bound the variance is similar --- the key point is
that for each fixed pair of collections of shifts
$\beta_{1}, \ldots, \beta_{r}$ and $\beta_{1}', \ldots, \beta_{r}'$
the corresponding main term, arising from pairs of ``generic'' subsets
$J_{1}, \ldots, J_{r}$ and $J_{1}', \ldots, J_{r}'$, can be evaluated
exactly as before. To bound the contribution from non-generic
choices of subsets, we can use the previous argument to obtain a
bound that is only worse by a factor of
$N_0^{2r}$.
\end{proof}

\section{Concluding the Proof}
\label{sec:concluding-proof}
First fix a compactly supported Schwartz function $f$.   By
Theorem~\ref{thm:corr2}, asymptotics for the average and the second
moment of the deterministic sums is given, up to negligible errors, by
the corresponding average and second moment of the random
 model (cf. \eqref{eq:randomeq1} \eqref{eq:randomeq2}.)
To pass to correlations over distinct angles we use a standard 
combinatorial sieving argument (cf. \cite[Section 4]{Rudnick-Sarnak}, note that \eqref{eq:randomeq1} and \eqref{eq:randomeq2} hold for arbitrary Schwartz functions $f$), which allows us to match the average and second moment of the $r$-correlation of \textit{distinct} angles of the deterministic sum to those of the corresponding random model $\mathcal R_{r,n_0}^*$ (as given in \eqref{eq:beta-shift}.) 
The 
first and second moments of the random model $\mathcal R_{r,n_0}^*$ are evaluated in
Section~\ref{sec:poiss-corr-hybr}
(cf. Lemma~\ref{lem:hybrid-moments}.)  Further, since the second
moment equals, up to smaller order terms, the square of the first
moment, the fluctuations around the mean is $o(1)$ for a full density
subsequence of $n \in \mathcal S$ (recall \eqref{eq:fixedn0}.)

Thus, once we remove a zero density subset of elements in
$\mathcal{S}$, the $r$-level correlation with respect to the fixed
Schwartz function $f$ is Poissonian. To show that the same holds for
all compactly supported Schwartz functions, we may take a countable and
dense (say in the $C^{\infty}$-norm) collection of compactly supported
Schwartz functions, and a standard diagonalization argument then gives
Poisson correlations (for all $r$) for some full density
sub-subsequence of elements in $\mathcal{S}$ {\em provided} that the
correlation functionals are continuous with respect to the
$C^{\infty}$-norm. This in turn can be seen as follows: if $f$ is
supported in some ball of radius $\delta$, then
$\limsup_{N\to \infty}\E(|\mathcal{R}_{r}^*(F_{N})|) \ll_{\delta}
\lVert f \rVert_{\infty}$.

Finally, since the $r$-level correlations are Poissonian along a full
density
subset of $\mathcal{S}$, Theorems~\ref{thm:poisson} and
\ref{thm:poissonjoint} are easily deduced using a combinatorial
argument (e.g. see \cite[Appendix A]{kurlberg-rudnick}.)


\begin{thebibliography}{10}

\bibitem{ABCZ}
V.~Augustin, F.~P. Boca, C.~Cobeli, and A.~Zaharescu.
\newblock The {$h$}-spacing distribution between {F}arey points.
\newblock {\em Math. Proc. Cambridge Philos. Soc.}, 131(1):23--38, 2001.

\bibitem{BCZ00}
F.~P. Boca, C.~Cobeli, and A.~Zaharescu.
\newblock Distribution of lattice points visible from the origin.
\newblock {\em Comm. Math. Phys.}, 213(2):433--470, 2000.

\bibitem{BCZ01}
F.~P. Boca, C.~Cobeli, and A.~Zaharescu.
\newblock A conjecture of {R}. {R}. {H}all on {F}arey points.
\newblock {\em J. Reine Angew. Math.}, 535:207--236, 2001.

\bibitem{boca-zaharescu}
F.~P. Boca and A.~Zaharescu.
\newblock The correlations of {F}arey fractions.
\newblock {\em J. London Math. Soc. (2)}, 72(1):25--39, 2005.

\bibitem{bk1}
E.~B. Bogomolny and J.~P. Keating.
\newblock Random matrix theory and the {R}iemann zeros. {I}. {T}hree- and
  four-point correlations.
\newblock {\em Nonlinearity}, 8(6):1115--1131, 1995.

\bibitem{bk2}
E.~B. Bogomolny and J.~P. Keating.
\newblock Random matrix theory and the {R}iemann zeros. {II}. {$n$}-point
  correlations.
\newblock {\em Nonlinearity}, 9(4):911--935, 1996.

\bibitem{chandee-lee}
V.~Chandee and Y.~Lee.
\newblock {$n$}-level density of the low-lying zeros of primitive {D}irichlet
  {$L$}-functions.
\newblock {\em Adv. Math.}, 369:107185, 70, 2020.

\bibitem{chaubey-yesha}
S.~Chaubey and N.~Yesha.
\newblock The distribution of spacings of real-valued lacunary sequences modulo
  one.
\newblock available at arXiv:2108.00431, 2021.

\bibitem{chung}
K.-L. Chung.
\newblock An estimate concerning the {K}olmogoroff limit distribution.
\newblock {\em Trans. Amer. Math. Soc.}, 67:36--50, 1949.

\bibitem{cohen-comp-alg-nt-book}
H.~Cohen.
\newblock {\em A course in computational algebraic number theory}, volume 138
  of {\em Graduate Texts in Mathematics}.
\newblock Springer-Verlag, Berlin, 1993.

\bibitem{Coleman}
M.~D. Coleman.
\newblock A zero-free region for the {H}ecke {$L$}-functions.
\newblock {\em Mathematika}, 37(2):287--304, 1990.

\bibitem{conrey-keating}
B.~Conrey and J.~P. Keating.
\newblock Moments of zeta and correlations of divisor-sums: {V}.
\newblock {\em Proc. Lond. Math. Soc. (3)}, 118(4):729--752, 2019.

\bibitem{conrey-snaith}
J.~B. Conrey and N.~C. Snaith.
\newblock In support of {$n$}-correlation.
\newblock {\em Comm. Math. Phys.}, 330(2):639--653, 2014.

\bibitem{entin-roditty-gershon-rudnick}
A.~Entin, E.~Roditty-Gershon, and Z.~Rudnick.
\newblock Low-lying zeros of quadratic {D}irichlet {L}-functions,
  hyper-elliptic curves and random matrix theory.
\newblock {\em Geom. Funct. Anal.}, 23(4):1230--1261, 2013.

\bibitem{erdos-hall}
P.~Erd\H{o}s and R.~R. Hall.
\newblock On the angular distribution of {G}aussian integers with fixed norm.
\newblock volume 200, pages 87--94. 1999.
\newblock Paul Erd\H{o}s memorial collection.

\bibitem{feller-probability-vol2}
W.~Feller.
\newblock {\em An introduction to probability theory and its applications.
  {V}ol. {II}}.
\newblock John Wiley \& Sons, Inc., New York-London-Sydney, second edition,
  1971.

\bibitem{Granville-Koukoulopoulos}
A.~Granville and D.~Koukoulopoulos.
\newblock Beyond the {LSD} method for the partial sums of multiplicative
  functions.
\newblock {\em Ramanujan J.}, 49(2):287--319, 2019.

\bibitem{Hooley1}
C.~Hooley.
\newblock On the difference of consecutive numbers prime to {$n$}.
\newblock {\em Acta Arith.}, 8:343--347, 1962/63.

\bibitem{Hooley2}
C.~Hooley.
\newblock On the difference between consecutive numbers prime to {$n$}. {II}.
\newblock {\em Publ. Math. Debrecen}, 12:39--49, 1965.

\bibitem{Hooley3}
C.~Hooley.
\newblock On the difference between consecutive numbers prime to {$n$}. {III}.
\newblock {\em Math. Z.}, 90:355--364, 1965.

\bibitem{bingrong}
B.~Huang.
\newblock Sup-norm and nodal domains of dihedral {M}aass forms.
\newblock {\em Comm. Math. Phys.}, 371(3):1261--1282, 2019.

\bibitem{IK}
H.~Iwaniec and E.~Kowalski.
\newblock {\em Analytic number theory}, volume~53 of {\em American Mathematical
  Society Colloquium Publications}.
\newblock American Mathematical Society, Providence, RI, 2004.

\bibitem{katai-kornyei}
I.~K\'{a}tai and I.~K\"{o}rnyei.
\newblock On the distribution of lattice points on circles.
\newblock {\em Ann. Univ. Sci. Budapest. E\"{o}tv\"{o}s Sect. Math.}, 19:87--91
  (1977), 1976.

\bibitem{katz-sarnak}
N.~M. Katz and P.~Sarnak.
\newblock Zeroes of zeta functions and symmetry.
\newblock {\em Bull. Amer. Math. Soc. (N.S.)}, 36(1):1--26, 1999.

\bibitem{kurlberg-rudnick}
P.~Kurlberg and Z.~Rudnick.
\newblock The distribution of spacings between quadratic residues.
\newblock {\em Duke Math. J.}, 100(2):211--242, 1999.

\bibitem{attainable}
P.~Kurlberg and I.~Wigman.
\newblock On probability measures arising from lattice points on circles.
\newblock {\em Math. Ann.}, 367(3-4):1057--1098, 2017.

\bibitem{marklof-strombergsson}
J.~Marklof and A.~Str\"{o}mbergsson.
\newblock The distribution of free path lengths in the periodic {L}orentz gas
  and related lattice point problems.
\newblock {\em Ann. of Math. (2)}, 172(3):1949--2033, 2010.

\bibitem{mason-snaith}
A.~M. Mason and N.~C. Snaith.
\newblock Orthogonal and symplectic {$n$}-level densities.
\newblock {\em Mem. Amer. Math. Soc.}, 251(1194):v+93, 2018.

\bibitem{montgomery}
H.~L. Montgomery.
\newblock The pair correlation of zeros of the zeta function.
\newblock In {\em Analytic number theory ({P}roc. {S}ympos. {P}ure {M}ath.,
  {V}ol. {XXIV}, {S}t. {L}ouis {U}niv., {S}t. {L}ouis, {M}o., 1972)}, pages
  181--193, 1973.

\bibitem{niven}
I.~Niven.
\newblock {\em Irrational numbers}.
\newblock The Carus Mathematical Monographs, No. 11. The Mathematical
  Association of America. Distributed by John Wiley and Sons, Inc., New York,
  N.Y., 1956.

\bibitem{rotman-advanced-modern-algebra-book}
J.~J. Rotman.
\newblock {\em Advanced modern algebra}.
\newblock Prentice Hall, Inc., Upper Saddle River, NJ, 2002.

\bibitem{Rudnick-Sarnak0}
Z.~Rudnick and P.~Sarnak.
\newblock The {$n$}-level correlations of zeros of the zeta function.
\newblock {\em C. R. Acad. Sci. Paris S\'{e}r. I Math.}, 319(10):1027--1032,
  1994.

\bibitem{Rudnick-Sarnak}
Z.~Rudnick and P.~Sarnak.
\newblock Zeros of principal {$L$}-functions and random matrix theory.
\newblock volume~81, pages 269--322. 1996.
\newblock A celebration of John F. Nash, Jr.

\bibitem{rudnick-zaharescu}
Z.~Rudnick and A.~Zaharescu.
\newblock The distribution of spacings between fractional parts of lacunary
  sequences.
\newblock {\em Forum Math.}, 14(5):691--712, 2002.

\bibitem{sarnak-poisson-pair}
P.~Sarnak.
\newblock Values at integers of binary quadratic forms.
\newblock In {\em Harmonic analysis and number theory ({M}ontreal, {PQ},
  1996)}, volume~21 of {\em CMS Conf. Proc.}, pages 181--203. Amer. Math. Soc.,
  Providence, RI, 1997.

\bibitem{smirnov}
N.~V. Smirnov.
\newblock Approximate laws of distribution of random variables from empirical
  data.
\newblock {\em Uspehi Matem. Nauk}, 10:179--206, 1944.

\bibitem{Tenenbaum}
G.~Tenenbaum.
\newblock {\em Introduction to analytic and probabilistic number theory},
  volume~46 of {\em Cambridge Studies in Advanced Mathematics}.
\newblock Cambridge University Press, Cambridge, 1995.
\newblock Translated from the second French edition (1995) by C. B. Thomas.

\bibitem{Titchmarsh}
E.~C. Titchmarsh.
\newblock {\em The theory of the {R}iemann zeta-function}.
\newblock The Clarendon Press, Oxford University Press, New York, second
  edition, 1986.
\newblock Edited and with a preface by D. R. Heath-Brown.

\end{thebibliography}

\end{document}